\documentclass{amsart}

\usepackage{amsthm, amssymb,amsmath}
\usepackage{mathscinet}
\usepackage[utf8]{inputenc}
\usepackage{hyperref}
\usepackage{graphicx}
\usepackage{subcaption}
\usepackage{dsfont}
\usepackage{rotating}

\newtheorem{theorem}{Theorem}[section]

\newtheorem{proposition}[theorem]{Proposition}
\newtheorem{lemma}[theorem]{Lemma}
\newtheorem{corollary}[theorem]{Corollary}
\newtheorem*{definition*}{Definition}
\newtheorem*{theorem*}{Theorem}

\renewcommand{\P}{\mathbb{P}}

\newcommand{\V}{\mathcal{V}}

\newcommand{\R}{\mathbb{R}}
\newcommand{\Z}{\mathbb{Z}}
\newcommand{\N}{\mathbb{N}}

\renewcommand{\S}{\mathcal{S}}
\newcommand{\F}{\mathcal{F}}
\newcommand{\X}{\mathcal{X}}

\newcommand{\nil}{\text{Nil}}

\newcommand{\e}{\varepsilon}

\newcommand{\G}{\Gamma}
\newcommand{\g}{\gamma}

\newcommand{\pp}{\partial}
\newcommand{\la}{\lambda}
\newcommand{\La}{\Lambda}

\newcommand{\aaa}{\mathfrak a}

\newcommand{\s}{\sigma}

\newcommand{\x}{\times}

\newcommand{\Om}{\Omega}

\newcommand{\U}{\mathcal U}
\newcommand{\CC}{\mathcal C}

\newcommand{\M}{\mathcal M}

\DeclareMathOperator{\Gr}{Gr}
\DeclareMathOperator{\SL}{SL}

\DeclareMathOperator{\id}{Id}

\newcommand{\dist}{\operatorname{dist}}

\begin{document}

\title{Dimension gap  and variational principle for  Anosov representations}
\author{Fran\c cois Ledrappier and Pablo Lessa}
\address{Fran\c cois Ledrappier, Sorbonne Universit\'e, UMR 8001, LPSM, Bo\^{i}te Courrier 158, 4, Place Jussieu, 75252 PARIS cedex
05, France,} \email{fledrapp@nd.edu}
\address{ Pablo Lessa,  CMAT, Facultad de Ciencias, Iguá 4225,  11400 Montevideo, Uruguay}  \email{lessa@cmat.edu.uy}

\thanks{FL was partially supported by IFUM; PL thanks CSIC research project 389}

\maketitle
\begin{abstract} We consider a representation of a finitely generated group \( \G\) in \(\SL (d, \R)\) that is Zariski dense and \(k\)-Anosov for at least two values of \(k\). We exhibit a gap for the Minkowski dimension of minimal sets for the action of \(\G\) on flags spaces. The proof uses a variational principle for the action on partial flags.
\end{abstract}

\section{Introduction}

\subsection{Main results}

Let \((\Gamma ,\S)\) be a finitely generated  group with a symmetric set of generators \(S\). 

For \(d \ge 2\), consider \(P\) a subset of \(\lbrace  1,\ldots, d-1\rbrace\).
A representation  \(\rho: \Gamma \to \SL(d,\R)\)
 is called  a {\it{\(P\)-Anosov}} representation if there exists \(c > 0\) such that
\[\frac{s_{p}(\rho(\gamma))}{s_{p+1}(\rho(\gamma))} > c\exp(c|\gamma|),\]
for all \(\gamma \in \Gamma\) and \(p \in P\). Here \(|\gamma|\) denotes word length in  \(\Gamma\) with respect to some fixed finite symmetric generating set \( S\) and, for  \(1\le i\le d, \, s_i(\g)\) denotes the \(i\)-th largest singular value of \(\g \in \SL(d,\R)\) with respect to the usual inner product on \(\R^d\). A representation is called {\it{Borel-Anosov}} if it is \( \lbrace  1,\ldots, d-1\rbrace\)-Anosov.

Anosov representations have been introduced by F. Labourie (\cite{Lab06}) and have been the subject of many studies (cf. in particular \cite{GGKW17}, \cite {KLP16}, \cite{KLP17}, \cite{KLP18a}, \cite{BPS19} and \cite{canary} for a recent survey). It is widely accepted that non-trivial Anosov representations present the right analog of convex cocompact representations in higher rank. Many properties are suggested by this analogy. In another direction, we present in this note a phenomenon which is purely higher rank: there is a dimension gap for minimal invariant subsets for the action of \( \rho (\G) \) on the  spaces of flags. The proof uses a new variational principle for the action of \( \rho (\G) \) on the spaces of partial flags (see below theorem \ref{theoremb}).

Let \( Q\) be a  non-empty subset of  \(\lbrace 1,\ldots, d-1 \rbrace\). We consider the space \( \F_Q\) of partial flags with signature \(Q,\) endowed with a rotationally invariant Riemannian metric.  Let us recall the definition of the Minkowski dimension of a set \( \La \subset \F_Q\): for \( \e >0,\) let \( N(\La, \e) \) be the covering number of the set \( \La \) by balls of radius \( \e\) in  the metric of \( \F_Q\). The {\it {Minkowski dimension }} \( \dim _M(\La ) \) is defined by 
\[ \dim _M (\La ) \; : =\; \limsup\limits_{\e \to 0} \frac{\log N(\La , \e)}{\log (1/\e)} .\]
Our result is 
\begin{theorem}[Dimension gap for minimal sets of Anosov representations]\label{main}
 With the above notations, let \((\G, S)\) be a  finitely generated group and  \(\La\) be a minimal invariant set for the action of \( \rho (\G) \) on \( \F _Q.\) Assume that the representation \(\rho\) is \( P\)-Anosov for some \( P \subset \lbrace 1,\ldots, d-1 \rbrace\) and such that \( \rho (\G)\)  is Zariski dense in \( \SL(d,\R).\) Then, with \( M := \# (Q \cap P),\) \[\dim_M(\Lambda) \le \dim(\F_Q) - \frac{M - 1}{2}.\] 
\end{theorem}

In the case \(d = 2\) the bound given by the above theorem is trivial.  And, indeed, uniform discrete subgroups of \(\SL(2,\R)\) have full limit set. A dimension gap for the limit set of classical Schottky groups was proved in \cite{Doy88}.  We note however, that  there exist uniform discrete (and thus convex co-compact) subgroups of isometries of 
 rank one symmetric spaces and thus there is no dimension gap for convex cocompact representations in rank one. On the other hand, there are many conditions that  ensure that the limit set is  a Lipschitz submanifold of the flag space (see the discussions in \cite{PSW19}). Our result is easy in those cases.

If \(\rho\) is \(P\)-Anosov then it is also \(P \cup P^\ast\)-Anosov for \(P^\ast  = \lbrace d-p: p \in P\rbrace\). Theorem \ref{main} gives a non-trivial upper bound for the dimension of any minimal invariant set  in \(\F_{Q}\) for all  \(Q\) with \(\#(Q\cap ( P \cup P^\ast)) \ge 2\). In particular, as soon as the representation is \(k\)-Anosov for some \( k \not = d/2\)   and Zariski dense, then the codimension of the limit set in the full flag space is at least 1/2.

There are a few simple reductions which allow us to state our core result.  Firstly, we observe that the bundle \(\pi _{Q, Q\cap P} : \F _Q \to \F_{Q\cap P} \) is smooth, so that for any closed  \( \La  \subset \F_{Q\cap P}, \) we have 
\[ \dim _M ((\pi _{Q, Q\cap P} )^{-1} (\La )) \; = \; \dim _M (\La) + \dim \F_Q - \dim \F _{Q\cap P} .\]
So, it suffices to show theorem \ref{main} for the case \( Q= P\). In that case, the unique minimal invariant subset of \(\F_{P}\) is the limit set \( \La _\rho.\)

 By considering \( \G /{\textrm {Ker}} \rho,\) we may assume that the representation \(\rho\) is faithful.  Furthermore, we may assume  that \(\Gamma\) is non-elementary hyperbolic since, on the one hand, an Anosov subgroup of \(\SL (d,\R) \) is hyperbolic (\cite{BPS19}, theorem 3.2) and, on the other hand,  the limit set of an Anosov representation of an elementary (i.e. virtually cyclic) hyperbolic group is finite. Taking all this into account, we  only have to prove 
\begin{theorem}[Dimension gap for limit sets of Anosov representations]\label{maintheorem}
Let \((\G,S)\) be a non-elementary, hyperbolic,  finitely generated group, \(  P \subset \lbrace  1,\ldots, d -1\rbrace\) with \( M:= \#P \ge 2 \) , \(\rho \) a faithful \(P\)-Anosov representation of \(\G\) in \( SL(d,\R) \) with \(\rho(\G)\) Zariski dense  in \( SL(d,\R) \) and \(\La _\rho \subset \F_{P} \) the limit set of the representation \( \rho.\)  Then, 
 \[\dim_M(\Lambda _\rho) \le \dim(\F_{P}) - \frac{M - 1}{2}.\] \end{theorem}

It seems natural to conjecture that a dimension gap will exist between the limit set and the corresponding space of flags, for any Zariski dense representation into a reductive algebraic  Lie group of rank \(r \ge 2\) which is Anosov with respect to at least two simple roots.

Let \(g\in \SL(d,\R)\).  The numbers \(\la_i (g) := \lim\limits _n 1/n \log s_i(g^n) \) are the logarithms of the moduli of the eigenvalues of \(g\). Apply the Oseledets theorem in the trivial case of one matrix 
\(g \in SL(d,\R).\) The Oseledets decomposition is obtained from the Jordan  decomposition by grouping together the spaces corresponding to the eigenvalues with  the same modulus. These moduli are the \( \exp \la _i (g).\) 

In our setting, if \( \g \in \G,\)  for all \( p \in P, \; \la _p (\rho(\g))  -\la _{p+1}(\rho(\g)) \geq c |\g| .\) 
For all \( \g\in \G\), we obtain a coarser decomposition by grouping together, for \( 1\leq k \leq M+1,\)  the Oseledets subspaces associated with 
\(\la _i(\rho (\g)) \) with \( \la _{p_{k-1}}(\rho (\g)) > \la _i(\rho (\g)) \ge \la _{p_{k}}(\rho (\g)).\)
\footnote{With the convention that  \(p_0=0,  \la_{p_0} =+\infty\) and \(  p_{M+1}:= d.\)}
 This associates to every \(\g \in \G \) an element \( \eta (\rho (\g))\) of the space  \(X_P\) of decompositions of \( \R^d = E_1 \oplus \ldots \oplus E_{M+1} ,\) with \( \dim E_k = p_k- p_{k-1},\) namely 
\[ \eta (\rho (\g)) \; = \; E_1(\rho (\g))  \oplus \ldots \oplus E_{M+1} (\rho (\g)) \] with the property  that  for all \( k = 1, \ldots , M+1,\) the space \(E_{k} (\rho (\g)) \) is generated by  the Jordan spaces of \( \rho(\g)\) such that the modulus \(\exp \la\) of the eigenvalue satisfies \[\exp  \la _{p_{k-1}}(\rho (\g)) >\exp  \la  \ge \exp \la _{p_{k}}(\rho (\g)).\]

 \begin{corollary}\label{eigenvectors}  Let \((\G,S)\) be a hyperbolic finitely generated group,  \(  P \subset \lbrace  1,\ldots, d-1 \rbrace\) with  \( M:= \# P .\)
With the above notations,  if  \(\rho \) is a \(P\)-Anosov, Zariski dense representation of  \(\G\) in \( \SL(d,\R) ,\) set \(\Om _\rho \subset X_P \) the closure of the set \( \{ \eta (\rho (\g)); \g \in \G\}.\) Then, 
 \[\dim_M(\Om_\rho) \le \dim X_P -M +1.\] \end{corollary}
 
 Indeed, corollary \ref{eigenvectors} is trivial for \( d= 2\) and  for \( M=1\). It is trivial as well if \( \G\) is elementary. For \( d \ge 3\) and \( M\ge 2,\) we show in section  \ref{splitting} why  corollary  \ref{eigenvectors}   follows from theorem \ref{maintheorem}.

  The first  non-trivial case for our results  is when \( d=3\) and \( P= \{1,2\}.\) So, let \(\rho \) be a Borel-Anosov representation of the  finitely generated non-elementary group \((\G, S)\) in \( \SL(3,\R) \)  and we can consider the limit set \( \Lambda _\rho \) of the action of \(\rho(\G) \) on the space \(\F\) of complete flags in \(\R^3.\) Besides, for all  \(\g \in \G\), the matrix \(\rho(\g)\)  admits three distinct eigenspaces written \( (E_1 (\rho( \g)), E_2 (\rho( \g)), E_3 (\rho( \g)))\) in the order of the absolute values of the eigenvalues. Let \( \Om _\rho \) denote the closure of the \( \{(E_1 (\rho( \g)), E_2 (\rho( \g)), E_3 (\rho( \g))), \g \in \G \} \) in \(\P^2\x\P^2\x\P^2\). We obtain, from theorem \ref{maintheorem} and corollary \ref{eigenvectors},
  \begin{corollary}\label{2.5, 5}\[\dim_M(\Lambda_\rho) \le \frac{5}{2} < 3 = \dim(\F), \; \dim_M(\Om_\rho) \le  5 < 6 = \dim (\P^2\x\P^2\x\P^2).\] \end{corollary}
 Let \(\G\) be a surface group.  For the Hitchin component, the dimension of the projective limit set is 1 (see \cite{Lab06}) and our  result is not new. It is  is new and unexpected, as far as the authors are aware, for the Zariski dense representations that lie in the Barbot  component, i.e. the same connected  component as  the Teichm\"uller times Identity representations.

\subsection{Strategy of the proof of theorem \ref{maintheorem}}

\subsubsection{Random walk entropy}

Let \(\M\) be the set of probability measures \(\mu\) on \(\Gamma\) with countable support and  with finite first moment, meaning 
\[\sum\limits_{\gamma \in \Gamma}\mu(\gamma) |\gamma| < +\infty,\]
 such that the semi-group \(\Gamma_\mu\) generated by the support of \(\mu\) is non-elementary (i.e. contains two independent loxodromic elements).
 
  It is well known that if \(\mu \in \M\) then \(\mu\) has finite Shannon entropy, meaning
\[H(\mu) = -\sum\limits_{\gamma \in \Gamma}\mu(\gamma)\log(\mu(\gamma)) < +\infty.\]

The random walk entropy \(h_\mu\) of \(\mu \in \M\) is defined by
\begin{equation}\label{rwentropy} h_\mu = \lim\limits_{n \to +\infty}\frac{1}{n}H(\mu^{\ast n}),\end{equation}
where \(\mu^{\ast n}\) is the \(n\)-fold convolution of \(\mu\) with itself.

\subsubsection{Lyapunov exponents}
Denote \(\aaa^+  \) the cone \[ \aaa^+ := \lbrace a \in \R^{d}: a_1 \ge \dots \ge a_d, \sum a_i = 0\rbrace.\] For \( g \in \SL(d,\R) , \log S(g) \in \aaa ^+,\) where  \[\log S(g): = (\log s_1(g),\ldots, \log s_d(g)). \]
The Lyapunov exponents \( \{ \lambda_i(\rho _\ast\mu) \} \in \aaa ^+\) induced by \(\rho\), are defined by
\begin{equation}\label{exponents} \lambda(\rho_*\mu):= \{ \lambda_i(\rho _\ast \mu) \} = \lim\limits_{n \to +\infty} \frac{1}{n}\sum\limits_{\gamma \in \Gamma}\mu^{\ast n}(\gamma)\log\left(S(\rho(\gamma))\right).\end{equation}

The limits exist by Fekete lemma applied to \( \{\sum\limits_{j\le i}\sum\limits_{\gamma \in \Gamma}\mu^{\ast n}(\gamma)\log\left(s_j(\rho(\gamma))\right) \}_{n \in \N}.\)

\subsubsection{Lyapunov and Falconer dimension}\label{Lyadimension}

We say a pair \((i,j), 0<i<j \le d\) is separated by \(p \in P\) if \(i \le p < j\).  For each \(p \in P\) let \(S_p\) be the set of pairs separated by \(p\), and define \(S(P) = \bigcup\limits_{p \in P}S_p\). On the cone \(\aaa ^+\) we define the roots
\[\alpha_{i,j}(a) = a_i - a_j,\]
for \(i < j\) and notice that they are non-negative on \(\aaa ^+\). Lyapunov and Falconer dimensions are special values of Lyapunov and Falconer functionals on \(\aaa^+.\)

We  define for \(h \ge 0\) the Lyapunov functional on \(\aaa ^+\) by 
\begin{align*}L^P_h(a) = &\; \text{maximum of } \sum\limits_{(i,j) \in S(P)}r_{i,j}
\\ &\text{ subject to }0 \le r_{i,j} \le 1
\; \text{ and } \sum\limits_{(i,j) \in S(P)} r_{i,j}\alpha_{i,j}(a) \le h.
\end{align*}

For \(r \ge 0\) we also define the Falconer functional on \(\aaa ^+\) by
\begin{align*}F^P_r(a) = &\; \text{minimum of } \sum\limits_{(i,j) \in S(P)}r_{i,j}\alpha_{i,j}(a)
\\ &\text{ subject to }0 \le r_{i,j} \le 1
\; \text{ and } \sum\limits_{(i,j) \in S(P)} r_{i,j} \ge r.
\end{align*}

\begin{lemma}[Duality between Falconer and Lyapunov functionals]\label{lyapunovvsfalconer}
Given \(a \in \aaa^+\), for all \(r,h \ge 0\) one has \(F_r^P(a) \le h\) if and only if \(L_h^P(a) \ge r\).

Furthermore, let \(d(a)\) be the number of pairs \((i,j) \) in \(S(P)\) with  \(\alpha_{i,j}(a) = 0\). 
Then, \( F^P_r(a) = 0 \) for \( r \in [0,d(a)]\) and  \(r \mapsto F^P_r(a)\) is an increasing homeomorphism from \([d(a),D]\) to \([0,F_{D}^P(a)]\) where \(D = \# S(P)=\dim (\F_P)\), and \(h \mapsto L_h^P(a)\) is its inverse.
\end{lemma}
\begin{proof}
  To establish the first claim observe that directly from the definitions, both conditions \(F_r^P(a) \le h\) and \(L_h^P(a) \ge r\) are equivalent to there being some choice of \(r_{i,j} \in [0,1]\) for \((i,j) \in S(P)\) such that
  \[\sum\limits_{(i,j) \in S(P)} r_{i,j}\alpha_{i,j}(a) \le h\text{ and }\sum\limits_{(i,j) \in S(P)} r_{i,j} \ge r.\]
Moreover, in computing \(F^P_r(a),\) we can always take \( r_{i,j} = 1 \) for the pairs  \((i,j) \) in \(S(P)\) with  \(\alpha_{i,j}(a) = 0\). It follows that  \( F^P_r(a) = 0 \) for \( r \in [0,d(a)]\). Then,  there is \(\e >0\) with  \(\alpha_{i,j}(a) > \e\) for all the  other pairs \((i,j)\). It follows immediately that \(F^P_{r+s}(a) >F^P_r(a) + \e s\) for all \(d(a) \le r < r+s \le D\).  Hence \(r \mapsto F^P_r(a)\) is  increasing on \([d(a), D]\).  Since \(F_{d(a)}^P(a) = 0\), \(r \mapsto F^P_r(a)\) it is an increasing  homeomorphism as claimed.

  Similarly, picking \(\e > 0\) small enough so \(\alpha_{i,j}(a) < \e^{-1}\) for all \((i,j) \in S(P)\), it follows that \(L^P_{h + s}(a) > L^P_h(a) + \e s\) for all \(0 \le h < h+s < F_D^P(a)\).  Hence, \(h \mapsto L_h^P(a)\) is an increasing homeomorphism as well.
  
  Let \(f:I \to J\) and \(g:J \to I\) be the two homeomorphisms under consideration where \(I = [d(a),D]\) and \(J = [0,F_D^P(a)]\).  By our first claim we have \(g(f(r)) \ge r\) for all \(r \in I\) and \(f(g(h)) \le h\) for all \(h \in J\).  However, substituting \(r = f^{-1}(h)\), the first inequality implies \(g(h) \ge f^{-1}(h)\) for all \(h \in J\), from which we obtain \(f(g(h)) \ge h\) for all \(h \in J\).  Hence, we have \(f(g(h)) = h\) for all \(h \in J\) as required.
\end{proof}

 Following \cite{DO80} and \cite{KY79}, we define the Lyapunov dimension of \(\rho_*\mu\) on \(\F_P\) by
 \[\dim_{LY}^P(\rho_*\mu) \;  := \; L_{h_\mu} (\lambda(\rho_*\mu)). \]
 
In the spirit of  \cite{Fal88}, we define the Falconer dimension \(\dim_F^P(\rho)\) of \(\rho\) relative to \(\F_P\) as the critical parameter \(r \ge 0\) for convergence of the series
\begin{equation}\label{falconerseries} \Phi _\rho^P (r ) \; := \; \sum\limits_{\gamma \in \Gamma} \exp (-(F_r^P\circ\log\circ S \circ \rho)(\gamma)).\end{equation}

\subsubsection{Proof of theorem \ref{maintheorem}}

Theorem \ref{maintheorem} is a direct consequence of the three following results. We assume that  \((\G,\S)\) is  a non-elementary, hyperbolic,  finitely generated group, \(  P \) a subset of \(\lbrace  1,\ldots, d -1\rbrace\) with \( M:= \#P \ge 2 \), \(\rho \) a faithful \(P\)-Anosov representation of \(\G\) in \( SL(d,\R) \) with \(\rho(\G)\) Zariski dense  in \( SL(d,\R) \) and we denote \(\La _\rho \subset \F_{P} \) the limit set of the action of  \( \rho (\G).\) 

\begin{theorem}[Upper bound on Lyapunov dimension]\label{theorema}
  For all \(\mu \in \M,\) one has 
  \[\dim_{LY}^P(\rho_*\mu) \le \dim(\F_P) - \frac{M - 1}{2}.\]
\end{theorem}

\begin{theorem}[Variational principle]\label{theoremb}
One has  \begin{equation} \label{Lyapunov} \dim_F^P(\rho) = \sup\limits_{\mu \in \M}\dim_{LY}^P(\rho_*\mu).\end{equation}
\end{theorem}

\begin{theorem}[Inequality between Minkowski and Falconer dimension]\label{theoremc}
 One has 
 \[\dim_M(\Lambda_\rho) \le \dim_F^P(\rho).\]
\end{theorem}

The key observation behind theorem \ref{theorema} is that the entropy \( h_\mu\) is realized as the Furstenberg entropy of  the action of \( \rho (\G)\) on the Grassmannian  \( \Gr(p,\R^d)\) for any \( p  \in P.\) This follows from the \(P\)-Anosov property and identification of the Poisson boundary and geometric boundary for \(\mu\)-random walks on \(\Gamma\) (\cite{Kai00}). The inequality in theorem \ref{theorema} is then a consequence of the  inequality between Furstenberg entropy and the sum of the relevant exponents (see \cite{LL23} and section \ref{prooftheorema}).

The fact that 
\begin{equation}\label{LyapleqFalc}  \dim_{LY}^P(\rho_*\mu) \; \leq \; \dim_F^P (G) \end{equation} is classical. We recall the proof in section \ref{proofLyap}.

The converse is due to  Yuxiang Jiao,  Jialun Li, Wenyu Pan and Disheng Xu (\cite{JLPX23}) in the case when the representation is Borel-Anosov.  In \cite{JLPX23}, given \( \e>0,\)  the authors  construct a free semi-group in \(\G\) that is rich enough that the uniform probability measure \( \mu _\e \) on the generators satisfies \(\dim_{LY}(\rho_*\mu _\e) \ge \dim_F(\rho) -\e.\) The complete dominated splitting is used to  control the almost additivity of the singular values.  
Here,  we consider a general subset \(P\) and we assume that the representation \(\rho(\G)\) is Zariski dense in  \( \SL(d,\R)\). In the spirit of  \cite{gouezel}, \cite{random} and \cite{JLPX23}, we prove in the appendix a general proposition about free subsemigroups  in hyperbolic groups (see the appendix for the definition of a quasi-geodesic set; see also \cite{yang} for a construction of free sub-semi-groups with convexity properties in a more general context.)
\begin{proposition}[=Proposition \ref{semigrouplemma}]
 Let \((\Gamma,\S)\) be a non-elementary hyperbolic group with a finite symmetric generator \(\S\).   Then there exists a semigroup \(S \subset \Gamma\), a finite set \(F \subset \Gamma\) and mappings \(L,R:\Gamma \to F\) with the following properties:
 \begin{enumerate}
  \item \(S\) generates \(\Gamma\) as a group,
  \item \(S\) is quasi-geodesic, and
  \item \(L(\gamma)\gamma R(\gamma) \in S\) for all \(\gamma \in \Gamma\).
 \end{enumerate}
\end{proposition}
Then, by Zariski density and \cite{AMS95}, there is a measurable complete splitting and we are  able to control the almost additivity of the singular values.

Finally, the proof of theorem \ref{theoremc} is done in section \ref{prooftheoremc}. Similarly to other  previous works (see \cite{Zha97}, \cite{BCH10}, \cite{PSW21}, \cite{GMT19}, \cite{FS20}), it rests  on covering the limit set in \( \partial \G \)  by shadows and counting the covering numbers of the images of these shadows in   \(\La _\rho \).

\subsection{Related facts} 
For convex cocompact groups in \({\textrm {SO}}(n,1),\) Dennis Sullivan proved in \cite{Sul79} that the Hausdorff dimension and the Minkowski dimension of the limit set coincide with the Poincar\'e exponent of the group (which is the Falconer  dimension in that case). This result can be extended to more general cases if one uses a conformally invariant  metric on the boundary (see \cite{Lin04}, \cite{DK22}). The problem is more delicate for the rotation invariant Riemannian metric (see e.g. \cite{Duf17} for the \({\textrm {SU}}(n,1)\) case).

Following the pioneer work of Falconer (\cite{Fal88}), there are many results comparing Hausdorff dimension and Falconer dimension for  affine IFS. The variational principle (\ref{Lyapunov}) for IFS can be found in \cite{Fal88} and  \cite{K04}. Under the form (\ref{Lyapunov}), it  is due to Ian D. Morris and Pablo Shmerkin (\cite{MS19}) in the case of dimension 2, Ian D. Morris and \c Ca\u{g}ri Sert (\cite{MS21}) in general. 

Given theorems \ref{theoremb} and \ref{theoremc}, a natural step to prove  \( \dim _H (\La_\rho) = \dim_M (\La _\rho) =  \dim _F^P (\rho) \) for \(P\)-Anosov representations is  finding  necessary conditions  on  a given  probability \(\mu \) on \(\G\) and a \( \rho _\ast (\mu)\)-stationary probability measure \(\nu \) on \( \F_{P}\) for having
\begin{equation}\label{dimension} \dim _{loc} \nu \; = \; \dim _{LY}^P (\rho _\ast \mu ) .\end{equation}
Here, \( \dim _{loc} \nu \) is the \(\nu \)-a.e. constant value of \( \lim\limits _{\e \to 0} \frac{\log \nu (B(x,\e))}{\log \e} \) (cf. \cite{LL23}). Relation (\ref{dimension}) for IFS is a famous problem, with many deep contributions (most relevant for us are \cite{Fal88}, \cite{JPS07}, \cite{Hoc14},   \cite{BHR19}, \cite{HR19}, \cite{FS22},  \cite{Rap22}, \cite{MS21}). For random walks on discrete subgroups of \(\SL (d,\R) \), relation (\ref{dimension}) is classical when \( d=2\) (see \cite{HS17} for the general non-discrete case). \cite{LL23a} prove  relation (\ref{dimension})  for the Hitchin component of the representation of a surface group in \( \SL (d,\R).\) In
\cite{LPX23},  Jialun Li, Wenyu Pan and Disheng Xu  prove relation (\ref{dimension})  for \( d=3\) when the representation \(\rho \) is Borel-Anosov and \( \rho (\G)\) is Zariski dense. They also observe that relation (\ref{dimension}) may be  wrong when the Zariski closure of \( \rho (\G) \) is conjugated to \( \SL(2, \R).\) See \cite{DS22} for the case of \( {\textrm {SU}}(2,1).\)

\section*{Acknowledgments}

The authors would like to thank Jairo Bochi, León Carvajales, Jialun Li, Rafael Potrie and Andrés Sambarino for many helpful conversations.

\section{Random walk entropy}

Let \(\G\) be a non-elementary finitely generated word hyperbolic group, \(P \subset \lbrace 1,\ldots, d-1 \rbrace\) with \( M:= \#P \ge 2 \), \(\rho \) a \(P\)-Anosov representation of \(\G\) in \( SL(d,\R) .\) 
We let \(\mathcal{M}\) denote the set of probability measures \(\mu\)  on \(\Gamma\) with finite first moment
\[\sum\limits_{\gamma \in \G} |\gamma| \mu(\gamma) < +\infty,\]
and such that the semi-group \(\Gamma_\mu\) generated by the support of \(\mu\) contains two independent loxodromic elements.

 In this section, we prove that the random walk entropy \(h_\mu \) given by (\ref{rwentropy}) coincides with the Furstenberg entropy (see below) of the stationary measure on the Gromov boundary of \(\Gamma\) and its images on adequate flag spaces.   We also relate the Falconer dimension and the random walk entropy to the growth indicator defined by \cite{Qui02a}.  To conclude the section, we prove corollary \ref{eigenvectors}.

\subsection{Furstenberg entropy}

Recall that a probability measure \(\nu\) on a compact space \(X\) on which \(\Gamma\)-acts continuously is said to be \(\mu\)-stationary where \(\mu\) is a probability measure on \(\Gamma\), if and only if
\[\nu = \sum\limits_{\gamma \in \Gamma}\mu(\gamma)\gamma_*\nu.\]

The Furstenberg entropy of a \(\mu\)-stationary measure is defined as
\[\kappa(\mu,\nu) = \sum\limits_{\gamma \in \Gamma} \mu(\gamma) \int\limits_{X} \log\left(\frac{d\gamma_*\nu}{d\nu}(\g x)\right)d\nu(x),\]
or \(+\infty\) if the Radon-Nikodym derivative in the integral does not exist for some \(\gamma\) in the support of \(\mu\). 

In this subsection we show that for \(\mu \in \M\), the Furstenberg entropy of the natural \(\mu\)-stationary measures on the Gromov boundary \(\partial\Gamma\), the Grasmannian manifolds of \(p\)-dimensional subspaces of \(\R^d\) for \(p \in P\), and the space of partial flags \(\F_P\) all coincide with the random walk entropy \(h_\mu\).

This particular feature of Anosov representations is useful since the Furstenberg entropy is what occurs in the dimension formulas of \cite{LL23}.

\subsubsection{Stationary measure on the Gromov boundary \(\partial \Gamma\)}

\begin{theorem}\label{furstenberg}
For each \(\mu \in \M\) there exists a unique \(\mu\)-stationary measure \(\nu_\mu\) on the Gromov boundary \(\partial \Gamma\) of \(\Gamma\), and its Furstenberg entropy is given by \(\kappa(\mu,\nu_\mu) = h_\mu\).
\end{theorem}
\begin{proof}
  If \(\g_{-1},\ldots, \g_{-n},\ldots\) are i.i.d. random elements in \(\Gamma\) with common distribution \(\mu\).  Then by \cite[Theorem 7.6]{Kai00} there exists a random limit point
  \[X = \lim\limits_{n \to +\infty}\g_{-1}\cdots \g_{-n} \in \partial\Gamma,\]
  almost surely, and the distribution \(\nu_\mu\) of \(X\) is the unique \(\mu\)-stationary measure on \(\partial \Gamma\).

  Furthermore, \((\partial \Gamma,\nu_\mu)\) is isomorphic to the Poisson boundary of \((\Gamma,\mu)\) by \cite[Theorem 7.7]{Kai00}.

  The Furstenberg entropy \(\kappa(\mu,\nu_\mu)\) is equal to the difference between \(h_\mu\) and the entropy of the random walk \(\g_1, \g_1\g_2,\ldots\) conditioned on \(X\) (see \cite[Corollary 2]{KV83}).  However, because \((\partial \Gamma,\nu)\) is the Poisson boundary this conditional entropy is zero (\cite[Theorem 4.3 and 4.5]{Kai00}) and therefore \(h_\mu = \kappa(\mu,\nu_\mu)\) as claimed.
\end{proof}

\subsubsection{Boundary maps}

Let \(p \in P \) and \(\gamma \in \Gamma\) be such that \(s_p(\rho(\gamma)) > s_{p+1}(\rho(\gamma))\).

There exists a unique \(p\)-dimensional subspace \(\xi^p(\gamma) \subset \R^d\) on which \(p\)-dimensional volume is most contracted by \(\rho(\gamma)^{-1}\).

We let \(\xi(\gamma) = (\{0\} \subset \xi^{p_1}(\gamma) \subset \ldots \subset  \xi^{p_{M}}(\gamma))\subset \R^d \in \F_P\) where \(P = \lbrace   p_1 < \cdots < p_{M} \rbrace\), whenever all \(\xi^p(\gamma)\) are well defined.

Since \(\rho\) is \(P\)-Anosov this is the case outside of a finite subset of \(\Gamma\).

We refer to \cite[Chapitre 7]{GdlH90} for basic properties of the Gromov compactification of \(\Gamma\).  We briefly recall that \(\Gamma\) with its word metric is a proper geodesic metric space.

The group \(\Gamma\) is word hyperbolic, so there exists \(\delta > 0\) such that for every geodesic triangle each side is contained in the \(\delta\)-neighborhood of the other two.

The Gromov boundary \(\partial \Gamma\) is the set of equivalence classes of word geodesic rays in \(\Gamma\), where two rays are equivalent if they are at bounded Hausdorff distance.

A basis of neigborhoods of  a point \(x \in \partial \Gamma\) in the Gromov compactification \(\overline{\Gamma} = \Gamma \cup \partial \Gamma\) is defined by taking for each \(C > 0\) the set
\[N(x,C) = \lbrace x\rbrace \cup  \lbrace y \in \overline{\Gamma}: \min\limits_{n} |\alpha_n| > C\text{ for all geodesics }\alpha\text{ joining }x\text{ and }y\rbrace.\]

\begin{proposition}\label{boundarymapstheorem}
  The map \(\xi\) defined above extends H\"older continuously to the Gromov boundary \(\partial \Gamma\).  Furthermore, \(\xi^p:\partial\Gamma \to \Gr(p,\R^d)\) is injective for each \(p \in P \setminus \{0, d\}\).
\end{proposition}
\begin{proof} 
  For each \(p \in P,\) the extension of \(\xi ^p \) to a continuous mapping from \(\pp \G\) to  \({\textrm {Gr}}(p ,\R^d )\) (the {\it{Cartan property }}) is due  to \cite{GGKW17} (see \cite{canary}, Proposition 30.3). Notice that if \(p \in P\) then \(\rho\) is also \((d-p)\)-Anosov.  From \cite[Proposition 4.9]{BPS19} one has that if \(x,y \in \partial \Gamma\) are distinct then \(\xi^p(x) \oplus \xi^{d-p}(y) = \R^d\).  Since \(\xi^{\min(p,d-p)}(x) \subset \xi^{\max(p,d-p)}(x)\), this implies, for \( p \not = 0,d,\)  that \(\xi^p(x) \neq \xi^p(y)\) so that \(\xi^p\) is injective as claimed. 
\end{proof}

\subsubsection{Dynamical stationary measures}

Fix \(\mu \in \M\),  let \(\g_{-1},\ldots, \g_{-n},...\) be i.i.d. random elements of \(\Gamma\) with common distribution \(\mu\), \(\chi_1 > \cdots > \chi_N\) be the distinct Lyapunov exponents of \(\mu\), and \(d_1,\ldots, d_N\) their multiplicities.
Let \(A = \lbrace  d_1, d_1+d_2,\ldots, d_1+\cdots + d_{N-1} \rbrace\).
By Oseledets theorem for each \(a \in A\) there is a random limit
\[U_a = \lim\limits_{n \to +\infty} \xi^a(\g_{-1}\cdots \g_{-n}),\]
almost surely.

The collection \(U\) of  \(U_a\) for \(a \in A\) is a random element in the space of flags of signature \(A\).   Its distribution, and the projections of its distribution to coarser partial flag spaces are what were called dynamical stationary measures in \cite{LL23}.   In particular since \(\lambda_p(\rho_*\mu) > \lambda_{p+1}(\rho_*\mu)\) for all \(p \in P\) we have that \(\xi_*\nu_\mu\) is the dynamical stationary measure on \(\F_P\), and \(\xi^p_*\nu_\mu\) is the dynamical stationary measure on the Grasmannian of  \(p\)-dimensional subspaces of \(\R^d\) for each \(p \in P\).

\subsubsection{Furstenberg entropy}

\begin{theorem}\label{entropytheorem}
  For each \(\mu \in \M\) one has \[h_\mu = \kappa(\mu,\xi_*\nu_\mu) = \kappa(\mu,\xi^p_*\nu_\mu),\] for all \(p \in P, p\not = 0,d\).
\end{theorem}
\begin{proof}
  This is immediate from the  injectivity of the maps \(\xi^p\) given by proposition \ref{boundarymapstheorem}.
\end{proof}

\subsection{Growth indicator function}
\subsubsection{Growth function and Falconer dimension}

Given \(a \in \aaa^+\) we define the growth indicator at \(a\) by 
\begin{equation}\label{growth}
\psi_\rho(a) =  \| a\| \inf\limits_{\CC \ni a} 
\limsup\limits_{T \to \infty } \frac{1}{T}\log \# \{ \g \in \G, \log \circ S\circ \rho(\g)) \in \CC,  ||\log \circ S\circ \rho(\g) || \leq  T\} \end{equation}
where the infimum is over all  open subcones \(\CC \subset \aaa^+\) that contain \(a\).  

\begin{theorem}[Quint, \cite{Qui02a}]\label{cone}
The growth indicator function \(\psi_\rho\) is concave and \( \lbrace \psi_\rho > 0\rbrace\) is the interior of the limit cone defined by
\[\mathcal{L}_\Gamma = \lim\limits_{t \to +\infty}\frac{1}{t}\log S(\Gamma),\]
where the limit is taken in  the Hausdorff topology on closed sets.
\end{theorem}

Recall that we defined the Falconer dimension \(\dim_F^P(\rho)\) as the critical value  of the series
  \(r \mapsto  \sum\limits_{\gamma \in \Gamma}\exp\left(-F^P_r\circ \log\circ S\circ \rho(\gamma)\right).\) We  have (see \cite[Lemma 4.2]{sambarino}, \cite[Lemma III.1.3]{Qui02a}):

\begin{lemma}\label{falconer2}
Assume \( r\ge 0 \) is such that there is \( a \in \aaa^+ \) inside the interior of the limit cone \(\mathcal{L}_\Gamma \) with \( \psi _\rho (a) > F^P_r (a). \) Then, 
\[ r \, \le \, \dim _F^P (\rho ).\]
\end{lemma}

\subsubsection{Proof of (\ref{LyapleqFalc})}\label{proofLyap}

The following inequality is essentially due to Guivarc'h. 

\begin{lemma}[Fundamental inequality]
  For all \(\mu \in \mathcal{M}\) one has
  \[h_\mu \le \psi_{\rho}(\lambda(\rho_*\mu)).\]
\end{lemma}
\begin{proof} Indeed, by the subadditive ergodic theorem applied to the shift space directing the random walk, for all \( \e >0\) and \(n\) large enough, there are at least \( \exp (n(h_\mu -\e)) \) elements \(\g\)  of \(\G\) with \(\log(S(\rho(\gamma)))\) \(n \e\)-close to \( n\lambda (\rho_*\mu).\) \end{proof}

Recall that, given \(\mu \in \mathcal{M}\) we defined the Lyapunov dimension by
\[\dim_{LY}^P(\rho_\ast \mu) = L_{h_\mu}(\lambda(\rho_*\mu)).\]
Relation  (\ref{LyapleqFalc}) is the following statement 
\[\sup\limits_{\mu \in \mathcal{M}}\dim_{LY}^P(\rho_\ast \mu) \le \dim_F^P(\rho).\].
\begin{proof}
Fix \(\mu \in \mathcal{M}\) and let \(r_0 = \dim_{LY}^P(\rho_\ast \mu) = L_{h_\mu}(\lambda(\rho_*\mu))\).
From the  duality between Lyapunov and Falconer functionals we have
  \(F_{r_0}(\lambda_\mu) = h_\mu .\)
  
  We have  \( h_\mu \le \psi_{\Gamma}(\lambda(\rho_{*}\mu)) \) by the fundamental inequality, i.e. \( F_{r_0}(\la(\rho_*\mu) ) \le \psi _\G (\la(\rho_*\mu)) .\)  Hence,
\( r_0 \le   \dim_F^P(\rho) \) by lemma \ref{falconer2}.
\end{proof}
\subsection{Eigenspace splitting}\label{splitting}

In this subsection, we show that theorem \ref{maintheorem} implies  corollary \ref{eigenvectors}. 

Consider, for \( \g \in \G,\) the \(P\)-decomposition \(  \eta (\rho (\g)) \; = \; E_1(\rho (\g))  \oplus \ldots \oplus E_{M-1} (\rho (\g)) \).
By the \(P\)-Anosov property,   the angles between \( E_1(\rho (\g))  \oplus \ldots \oplus E_{k}(\rho (\g)) \) and \( E_{k+1}(\rho (\g))  \oplus \ldots \oplus E_{M-1}(\rho (\g)) \) are uniformly bounded from below by some positive number for all  \(1\le k \le M-2\) and for all \( \g\) but a finite number.

Recall that we denote, for all \(\g\) but a finite number, \(\xi(\gamma) = \{0\} \subset \xi^{p_1}(\gamma) \subset \ldots \subset  \xi^{p_{M-2}}(\gamma)\subset \R^d  \in \F_P\). In the same way,  \(\xi(\gamma^{-1}) = \{0\} \subset \xi^{d-p_{M-2}}(\gamma^{-1}) \subset \ldots \subset  \xi^{d-p_1}(\gamma^{-1})\subset \R^d  \in \F_{P^\ast}\). It follows from \cite{GGKW17}, lemma 2.26\footnote{The argument in the non-hyperbolic case goes back to  David Ruelle's proof of Oseledets theorem (\cite{Rue79}, cf. \cite{Led84} proposition 3.2 or \cite{Sim15}, pages 141--142 for details).}
 that the unstable flag \(f(\rho (\g)) \in \F_P\), 
\[ f(\rho (\g)) := \{ 0\} \subset E_1(\rho (\g)) \subset \ldots \subset E_1 (\rho (\g)) \oplus \ldots \oplus E_{M-2}(\rho (\g))  \subset \R^d \]
and the stable flag \(f'(\rho (\g)) \in \F_{P^\ast}\),
\[ f'(\rho (\g)) := \{ 0\} \subset E_{M-1}(\rho (\g)) \subset \ldots \subset E_2 (\rho (\g)) \oplus \ldots \oplus E_{M-1}(\rho (\g) ) \subset \R^d \]  are given by \[ f(\rho (\g))  \; =\; \lim\limits _{n\to +\infty} \xi (\g^n), \quad  f'(\rho (\g))  \; =\; \lim\limits _{n \to - \infty} \xi (\g^n) .\]
In particular, \(f (\rho (\g)) \in \La_\rho \subset \F_P, \) and \(f'(\rho (\g))  \in \La_\rho ^\ast, \) where \( \La_\rho ^\ast \) is the limit set associated to the representation \(\rho\) in \( \F_{P^\ast}.\) Thus, \((f,f')\in 
(\La_\rho \x \La_\rho ^\ast)' \subset ( \F_{P} \x  \F_{P^\ast})',\) where \('\) indicates that the pairs of flags are in general position.

The set of pairs \( (f,f') \in (\F_P \x \F_{P^\ast})'\) such that the angles between the opposite partial spaces are bounded from below by some positive number form a compact subset of 
\(( \F_P \x \F_{P^\ast})'\)  and the  mapping that associates to \((f,f')\) the underlying decomposition is uniformly Lipschitz on that compact set. For almost all \( \g \in \G,\) the decomposition \(\eta (\rho (\g))\) is obtained by this mapping from \((f(\rho (\g)) \x f'(\rho (\g))\). Therefore the closure \( \Om _\rho \) of the set of splittings \( \{ \eta (\g), \g\in \G\} \)  is the image of a compact subset of \( (\La_\rho \x \La_\rho ^\ast)' \)  by a Lipschitz mapping and its Minkowski dimension  is at most the Minkowski dimension of  \( \La_\rho \x \La_\rho ^\ast \). Since the Minkowski dimension of a product is the sum of the Minkowski dimensions, corollary \ref{eigenvectors} follows from theorem \ref{maintheorem} applied to the representation \( \rho\) which is  both \(P \)-Anosov  and \( P^\ast\)-Anosov.

\section{Entropy gap and proof of theorem \ref{theorema}}\label{prooftheorema}

In this section we write \(\lambda_i\) for \(\lambda_i(\rho_*\mu)\). The purpose of this section is to show that for all \(\mu \in \M\), the random walk entropy \(h_\mu\) is far from the sum
\[\sum\limits_{(i,j) \in S(P)}\lambda_i - \lambda_j,\]
which would be needed to maximize the value of the Lyapunov dimension \(\dim_{LY}^P(\rho_*\mu)\).

This is the main estimate needed for the proof of theorem \ref{theorema}:
\begin{proposition}\label{mainlemma}
  For each \(\mu \in \M\) it holds that
  \[h_\mu + \frac{M- 1}{2}\left(\lambda_1 - \lambda_d\right) \le \sum\limits_{(i,j) \in S(P)}\lambda_i - \lambda_j.\]
\end{proposition}

\subsection{Proof of theorem \ref{theorema} assuming proposition \ref{mainlemma}}

Given \(\mu \in \M\) by definition we have, setting \( b_{i,j} := 1-a_{i,j},\)
\begin{equation}\label{lyapunovdimensionequation}
\dim_{LY}^P(\rho_*\mu) =  \# S(P) -\min \sum\limits_{(i,j) \in S(P)} b_{i,j},
\end{equation}
where the minimum is over all choices of \(b_{i,j}\) satisfying \(0 \le b_{i,j} \le 1\) and
\[\sum\limits_{(i,j) \in S(P)}b_{i,j} (\lambda_i - \lambda_j)\geq - h_\mu + \sum\limits_{(i,j) \in S(P)}(\lambda_i - \lambda_j) .\]
Since \(\lambda_1 - \lambda_d\) is the largest possible difference \( \lambda _i -\lambda _j,\)  it follows that
\begin{equation}\label{maininequality}
\dim_{LY}^P(\rho_*\mu) \le \# S(P) - \frac{\left(\sum\limits_{(i,j) \in S(P)}\lambda_i - \lambda_j\right) - h_\mu}{\lambda_1-\lambda_d}.\end{equation}

Letting \(P = \lbrace  p_1 < \cdots < p_{M} \rbrace\) one has
\[\# S(P) = p_1(d-p_1) + (p_2-p_1)(d-p_2) + \cdots + (p_{M} - p_{M-1})(d-p_{M}) = \dim(\F_{P}).\]

Substituting this into the inequality (\ref{maininequality}), and using proposition \ref{mainlemma} we obtain
\[\dim_{LY}^P(\rho_*\mu) \le \dim(\F_{P}) - \frac{M - 1}{2},\]
as claimed.

\subsection{Upper bounds for Furstenberg entropy}

In view of theorem \ref{entropytheorem} it will be useful to bound the Furtenberg entropy of the stationary measures \(\xi^p_*\nu_\mu\) for each \(p \in P\).

\begin{lemma}\label{entropyupperboundlemma}
  For each \(\mu \in \M\) and \(p \in P\) one has
  \[h_\mu \le \sum\limits_{(i,j) \in S_p}\lambda_i - \lambda_j.\]
\end{lemma}
\begin{proof}
  Let \(\chi_1 > \cdots > \chi_N\) be the distinct Lyapunov exponents and \(d_1,\ldots, d_N\) their multiplicities.   For each \(l = 1,\ldots, N\) let 
\(A_l = \lbrace i: \lambda_i = \chi_l\rbrace\).

Fix \(p \in P\), because \(\rho\) is \(p\)-Anosov one has \(\lambda_{p} > \lambda_{p+1}\) and therefore there exists \(k\) such that \(d_1 + \cdots + d_k = p\).

From \cite[Theorem 2.1]{LL23} the dynamical stationary measure \(\xi^p_*\nu_\mu\) on the Grasmannian of \(p\)-dimensional subspaces has Furstenberg entropy bounded by
\[\kappa(\mu,\xi^p_*\nu_\mu) \le \sum\limits_{l \le k < m}d_l d_m(\chi_l - \chi_m).\]

Noticing that \(\#A_l = d_l\) we obtain
\begin{align*}
  \sum\limits_{l \le k < m}d_l d_m(\chi_l - \chi_m) &= \sum\limits_{l \le k < m}\sum\limits_{i \in A_l,j \in A_m}\lambda_i - \lambda_j
                                                 \\ &= \sum\limits_{i \le p < j}\lambda_i - \lambda_j
                                                 \\ &= \sum\limits_{(i,j) \in S_p}\lambda_i - \lambda_j,
\end{align*}
which concludes the proof.
\end{proof}

\subsection{Proof of proposition \ref{mainlemma}}

Let \(M= \#P\), \(\alpha_k = \lambda_{k} - \lambda_{k+1}\) for each \(k = 1,\ldots, d-1\), and for \(A \subset P\) let \(S(A)\) be the set of \((i,j)\) such that \(1 \le i \le a < j \le d\) for some \(a \in A\).

We write
\begin{align*}
  M \left(\sum\limits_{(i,j) \in S(P)}\lambda_i - \lambda_j\right) &- \sum\limits_{p \in P}\sum\limits_{(i,j) \in S_p}\lambda_i - \lambda_j
                                                                         \\ &= \sum\limits_{k = 1}^{d-1} \left(Ma_k(P) - \sum\limits_{p \in P}b_k(p)\right)\alpha_k,
\end{align*}
where \(a_k(A)\) is the number of \((i,j) \in S(A)\) with \(i \le k < j\), and \(b_k(p) = a_k(\lbrace p\rbrace)\).

Since by lemma \ref{entropyupperboundlemma} 
\[M  h_\mu \le \sum\limits_{p \in P}\sum\limits_{(i,j) \in S_p}\lambda_i - \lambda_j,\]
it suffices to show that for each \(k = 1,\ldots, d-1\) one has
\[M a_k(P) - \sum\limits_{p \in P}b_k(p) \ge \frac{M(M-1)}{2}.\]

For this purpose we fix \(k\) and enumerate \(P = \lbrace  p_1,\ldots, p_{M}  \rbrace\) in such a way that
\[|p_1 -k| \ge |p_2 -k| \ge \cdots \ge |p_{M}-k|,\]
and set \(A_i = \lbrace p_1,\ldots, p_i\rbrace\) for \(i = 1,\ldots,M\).

The desired lower bound will follow from the following two claims:
\begin{enumerate}
  \item For each \(i = 1,\ldots, M\) one has \(a_k(A_i) \ge b_k(p_i)\).
  \item For each \(i = 1,\ldots, M-1\) one has \(a_k(A_{i+1}) \ge a_k(A_i) + 1\). 
\end{enumerate}

Claim (1) is trivial since \(\lbrace p_i\rbrace \subset A_i\). 

To establish claim (2) we first suppose that \(p_{i+1} \le k\).  If this is the case then \((p_{i+1},k+1) \in S(A_{i+1})\).   If \((p_{i+1},k+1) \in S(A_i)\) there would have to be some \(j \le i\) with \(p_{i+1} < p_j < k+1\) contradicting the choice of enumeration of \(P\).  Hence \(a_k(A_{i+1}) \ge a_k(A_i) + 1\) in this case.   Similarly, if \(p_{i+1} > k\) then \((k,p_{i+1}) \in S(A_{i+1}) \setminus S(A_i)\) and claim (2) follows.

To conclude the proof we observe that from claim (1) we have \(a_k(A_1) - b_k(p_1) \ge 0 = \frac{1 \times 0}{2}\).  

Using claims (1) and (2) we show inductively for \(i = 1,\ldots, M-1\) that
\begin{align*}(i+1)a_k(A_{i+1}) - \sum\limits_{j = 1}^{i+1}b_k(p_j) &= a_k(A_{i+1}) - b_k(p_{i+1}) + ia_k(A_{i+1}) - \sum\limits_{j = 1}^i b_k(p_j)
  \\ & \ge i + i a_k(A_i) - \sum\limits_{j = 1}^i b_k(p_j)
  \\ & \ge i + \frac{i(i-1)}{2} = \frac{(i+1)i}{2},
\end{align*}
which concludes the proof setting \(i = M-1\).

\section{Proof of theorem \ref{theoremb}}\label{prooftheoremb}
The inequality (\ref{LyapleqFalc}),  
\(\sup\limits_{\mu \in \mathcal{M}}\dim_{LY}^P(\rho_\ast \mu) \le \dim_F^P(\rho),\)
was proven in section \ref{proofLyap}. 

\subsection{Proof of \(  \dim_F^P(\rho)\le \sup\limits_{\mu \in \mathcal{M}}\dim_{LY}^P(\rho_\ast \mu) \)}
We recall that the limit cone of \(\rho\) is defined as
\[\mathcal{L}_{\rho} = \lim\limits_{T \to +\infty}\frac{1}{T} (\log \circ S \circ \rho) (\Gamma),\]
in the Hausdorff topology.  It is a closed convex cone in \(\aaa^+\).

By the \(P\)-Anosov property if \(a \in \mathcal{L}_{\rho} \setminus \lbrace 0\rbrace\) we obtain
\[\alpha_{i,j}(a) > 0\text{ for all }(i,j) \in S(P).\]

Assume that \(r < r' < \dim_F^P(\rho)\) so that the series (\ref{falconerseries}) diverges at \(r'\).  By compactness there exists \(a \in \mathcal{L}_{\rho}\setminus \lbrace 0\rbrace\) such that for all open cones \(\mathcal{C}\) containing \(a\) we have
\[\sum\limits_{\g \in \G: (\log \circ S \circ \rho)(\gamma) \in \mathcal{C}} \exp (-(F_{r'}^P\circ \log \circ S\circ \rho)(\gamma)) = +\infty.\]

We fix such a choice of \(a\) and notice that \(\psi(a) \ge  F_{r'}^P(a)\) by \cite[Lemma 4.2]{sambarino}. Then,  by lemma \ref{lyapunovvsfalconer}, we have \(\psi(a) > F_{r}^P(a)\).  We will prove:
\begin{proposition}\label{goodmeasureslemma}
There exists a sequence \(T_k \to +\infty\) and \(\mu_k \in \M\) such that \(h_{\mu_k}  =  T_k \psi(a) + o(T_k)\) and \(\lambda(\rho_*\mu_k) = T_k a + o(T_k)\) when \(k \to +\infty\).
\end{proposition}

Since \(\psi(a) > F_r^P(a)\) we have from lemma \ref{lyapunovvsfalconer} that \(L_{\psi(a)}^P(a) > r\).

Assuming proposition \ref{goodmeasureslemma} we obtain
\begin{align*}\dim_{LY}^P(\rho_*\mu_k) &= L_{h_{\mu_k}}^P(\rho_*\mu_k) = L_{T_k\psi(a) + o(T_k)}^P(T_k a + o(T_k))
\\ &= L_{\psi(a) + o(1)}^P(a + o(1)) = L_{\psi(a)}^P(a) + o(1) > r,\end{align*}
for  \(k\) large enough.  Which concludes the proof of the inequality and of  theorem \ref{theoremb}.

\subsection{Proof of proposition \ref{goodmeasureslemma}}

Fix \(\beta:\aaa \to \R\) linear such that \(\beta(a) = F_r^P(a)\), and a decreasing sequence of open cones \(\mathcal{C}_k, k = 1,2,3,\ldots\) whose intersection is \(\R_+ a\).

For a given \(k\) set 
\[\psi_k(a) := \limsup\limits_{T \to +\infty}\frac{1}{T}\log \#\lbrace \gamma \in \Gamma: (\log \circ S\circ \rho)(\gamma) \in \mathcal{C}_k \cap \lbrace \beta \le T \beta(a)\rbrace.\]
We have \( \psi (a) = \lim\limits _{k \to \infty } \psi_k (a) \) and
\[\psi_k(a) = \limsup\limits_{T \to +\infty}\frac{1}{T}\log \#\lbrace \gamma \in \Gamma: (\log \circ S\circ \rho)(\gamma) \in \mathcal{C}_k \cap \lbrace (T-1)\beta(a) \le \beta \le T \beta(a)\rbrace.\]

Hence, we may choose \(T_k \to +\infty\) such that the sets
\[A_k = \left\lbrace \gamma \in \G:  (\log \circ S \circ \rho)(\gamma) \in \mathcal{C}_k \cap \lbrace (T_k-1)\beta(a) \le \beta \le T_k \beta(a)\rbrace\right\rbrace,\]
satisfies
\[\psi(a) = \lim\limits_{k \to +\infty}\frac{1}{T_k}\log \# A_k.\]
Moreover, by the Anosov properties, there are constants \(C_1,C_2\) such that for  \(\g \in A_k, \, C_1T_k \le |\g| \le C_2T_k.\) So there exist \(\ell _k, \ell _k \to \infty \) as \( k \to \infty ,\) such that the sets
\[A'_k = \left\lbrace \gamma \in \G:  (\log \circ S \circ \rho)(\gamma) \in \mathcal{C}_k \cap \lbrace (T_k-1)\beta(a) \le \beta \le T_k \beta(a)\rbrace, |\g| = \ell _k  \right\rbrace,\]
still satisfies
\[\psi(a) = \lim\limits_{k \to +\infty}\frac{1}{T_k}\log \# A'_k.\]

We now fix a semi-group \(S\) of \(\Gamma\), a finite subset \(F\), and maps \(L,R:\Gamma \to F\) given by proposition \ref{semigrouplemma}.  Observe, that \(\rho(S)\) is Zariski dense since the Zariski closure of a semi-group is a group(see \cite{goldsheid-margulis}), and \(S\) generates \(\Gamma\).

Let \(B_k \subset S\) be defined by
\[B_k = \lbrace L(\gamma)\gamma R(\gamma): \gamma \in A'_k\rbrace.\]

Since, \(F\) is finite there exists \(C > 0\) such that
\begin{eqnarray*} \# B_k &\ge& C^{-1} \# A'_k, \\ |  |\g| - \ell _k | &\le & C {\textrm { for }} \g \in B_k \end{eqnarray*}
and fixing some norm on \(\aaa\) we have
\[\|\log\circ S \circ \rho(\gamma) - \log \circ S \circ \rho(L(\gamma)\gamma R(\gamma))\| \le C,\]
for all \(\gamma \in \Gamma\).

Fix a parameter \(\e> 0 \), we say a splitting \(\R^d = E_1 \oplus \cdots \oplus E_d\) into one dimensional subspaces is \(\e\)-non-degenerate if 
the angle between \(E_i\) and \(\bigoplus\limits_{j \neq i}E_j\) is at least \(\e\) for all \(i\).  We say two splittings \((E_1,\ldots, E_d), (E_1',\ldots,E_d')\) are \(\e\)-close if the angle between \(E_i\) and \(E_i'\) is at most \(\e\) for all \(i\).

We say that an element \(g \in \SL(d,\R)\) is \(\e\)-diagonalizable if it has \(d\) distinct real eigenvalues \(|\lambda_1| > \cdots > |\lambda_d|\) satisfying \(|\lambda_i|/|\lambda_{i+1}| > e^{\e}\) for \(i = 1,\ldots, d-1\), and the corresponding splitting into eigenspaces is \(\e\)-non-degenerate.

Since \(\rho(S)\) is Zariski dense we obtain from \cite[Theorem 6.8]{AMS95}:
\begin{lemma}[Abels-Margulis-Soifer]\label{abslemma}
There exists \(\e_0 > 0\) and a finite set \(S_0 \subset S\) such that for all \(\gamma \in \Gamma\) the set \(\rho(\gamma S_0)\) contains at least one \(\e_0\)-diagonalizable element.
\end{lemma}

For what follows we fix \(S_0\) and \(\e_0 > 0\) given by the previous lemma. We also fix \(\delta_0 > 0\) given by the following:
\begin{lemma}
  There exists \(\delta_0 > 0\) such that if \(B \subset \SL(d,\R)\) is a set of \(\e_0\)-diagonalizable elements whose eigenspace splittings are \(\delta_0\)-close then
  \(\gamma_1\cdots \gamma_n\) is \(\e_0/2\)-diagonalizable for all \(n\) and \(\gamma_1,\ldots, \gamma_n \in B\).
\end{lemma}
\begin{proof}
This follows directly from \cite[Lemma 2.17]{2021breuillard-sert}.
\end{proof}

We now refine the family \(B_k\) so as to obtain images under \(\rho\) which are diagonalizable with nearby splittings.
\begin{lemma}\label{bklemma}
 There exists a decreasing positive sequence \(\epsilon_k \downarrow 0\) and a sequence \( \ell '_k \uparrow \infty \) such that for all \(k\) large there exists \(C_k \subset B_k\cdot S_0\) with the following properties:
  \begin{enumerate}
    \item \(C_k\) freely generates a free semi-group in \(\Gamma\).
    \item For all \(\gamma \in C_k\) the element \(\rho(\gamma)\) is \(\e_0\)-diagonalizable.
    \item For all \(\gamma_1,\gamma_2 \in C_k\) the eigenspace splittings of \(\rho(\gamma_1)\) and \(\rho(\gamma_2)\) are \(\delta_0\)-close.
    \item One has \(\psi(a) = \lim\limits_{k\to +\infty}\frac{1}{T_k}\log \# C_k\).
    \item \(    |\g| = \ell '_k \) for  \( \g \in C_k .\) 
    \item For all \(\gamma \in C_k\) one has \(\|(\log\circ S\circ \rho)(\gamma) - T_k a\| \le \epsilon_k T_k\).
  \end{enumerate}
\end{lemma}
\begin{proof}
The space of \(\e_0\)-non-degenerate splittings is compact and hence we may cover it by some finite number \(N\) of subsets with the property that if two splittings belong to the same subset they are \(\delta_0\)-close.

From lemma \ref{abslemma} it follows that for all \(\gamma \in B_k\) there exists \(f_\gamma \in S_0\) such that \(\rho(\gamma f_\gamma)\) is \(\e_0\)-diagonalizable.  By the pigeonhole principle we may choose a set \(C'_k\) of at least \(\# B_k /N\) of the elements \(\gamma f_{\gamma}\) such that the eigenspace splitting of \(\rho\) applied to any two are \(\delta_0\)-close.

Hence, we have constructed \(C'_k\) with properties (2), (3) and (4) above.

Since \(F \cup S_0\) is finite there exists \(C > 0\) such that, for all \(k\) and \(\gamma \in A_k\),
\begin{eqnarray*} | |L(\gamma)\gamma R(\gamma) f_{L(\gamma)\gamma R(\gamma)}| - \ell _k | &\le &C  \; {\textrm { and }}\\
\|(\log\circ S\circ \rho)(L(\gamma)\gamma R(\gamma) f_{L(\gamma)\gamma R(\gamma)}) - (\log\circ S \circ \rho)(\gamma)\| &\le &C.\end{eqnarray*}  

By the pigeonhole principle again, we can find \(\ell' _k , | \ell '_k - \ell _k | \le C,\) such that  \(  C''_k : = \g \in C'_k, |\g | = \ell '_k  \)  satisfies properties (2), (3), (4) and (5). Moreover, 
if \(\gamma \in A_k\) then by definition we have
\[(\frac{1}{T_k}(\log\circ S \circ \rho)(\gamma)- a) \in \mathcal{C}_k \cap \lbrace \frac{T_k - 1}{T_k}\beta(a) \le \beta \le \beta(a)\rbrace,\]
hence setting \(\epsilon_k\) to be \(C/T_k\) plus the diameter of the set on the right-hand side, we have
\[\|(\log\circ S\circ \rho)(L(\gamma)\gamma R(\gamma) f_\gamma) - T_k a\| \le \epsilon_k T_k,\]
for all \(\gamma \in A_k\), which establishes property (6) for \(C'_k\) and \(C''_k.\) 

To establish (1) we notice that \(S\) is a quasi-geodesic semi-group. Therefore,  by Morse lemma,  there is some \(R > 0\) such that if \(\gamma_1,\ldots, \gamma_n, \eta_1,\ldots,\eta_n \in S\) and \(\gamma_1\cdots \gamma_n = \eta_1\cdots \eta_n\), then there is a geodesic ray \(\alpha = \{\alpha _n \}_{ n \ge 0},\) such that for all \( j \in [0, n],\)  \begin{equation}\label{tracking}\dist(\gamma_1\cdots \gamma_j, \alpha ), \, \dist (\eta_1\cdots \eta_j, \alpha ) < R.\end{equation}

We claim that if we assume that the elements of \( C_k \subset C''_k \) are \(6R\)-separated, equation (\ref{tracking}) forces \( \g _j = \eta _j \) for all \(  j \in [0, n].\)
Hence it suffices to refine \(C''_k\) so that it is a \(6R\)-separated set to obtain that it freely generates a free sub-semi-group of \(\Gamma\).   Doing this decreases the number of elements at most by a factor equal to the number of elements in a ball of radius \(6R\) in \(\Gamma\), so  property (4) still holds.

To prove the claim, we prove by induction on \(j\) that \( \g _i = \eta _i \) for \( i \le j.\) We have  \( \g _0 = \eta _0 = e.\) Assume that \( \g _i = \eta _i \) for \( i \le j.\) Then 
 \(\gamma_1\cdots \gamma_j = \eta_1\cdots \eta_j\) and there are \( N_1, N_2\) and \(N_3\) such that 
 \[ \dist (\gamma_1\cdots \gamma_j, \alpha _{N_1} ),\;\dist (\gamma_1\cdots \gamma_j\g_{j+1}, \alpha _{N_2} ) {\textrm { and }}\dist (\gamma_1\cdots \gamma_j\eta_{j+1}  , \alpha _{N_3} ) \le R.\]
 Setting \( \beta _t :=  (\gamma_1\cdots \gamma_j)^{-1} \alpha _{N_t} \) for \( t = 1,2,3,\) we have \(\beta _{N_1}, \beta _{N_2}\) and \( \beta _{N_3} \) aligned on the same geodesic and satisfying:
 \[ |\beta _1|, \dist (\g_{j+1}, \beta _2), \dist (\eta_{j+1}, \beta _3) < R \, {\textrm { and }} \, |\g_{j+1}|= |\eta _{j+1}| = \ell'_k .\]
 If \(\ell'_k\) is large enough, \( \beta _1 \) is outside the interval \( [\beta _2, \beta _3]\)  (see lemma \ref{boundeddistanelemma}) and \[ \dist (\beta _1, \beta _2), \dist (\beta _1, \beta _3) \; \in \; [\ell'_k -2R, \ell'_k +2R ].\]
 It follows that \(\dist (\beta _2, \beta _3) \le 4R\) and therefore that \(\dist (\alpha _{j+1} , \eta _{j+1} ) < 6R.\) Since \(C_k\) is \(6R\)-separated, \( \alpha _{j+1} = \eta _{j+1}\) as claimed.
\end{proof}

\begin{corollary}
If \(\mu_k\) is the uniform probability measure on \(C_k\) then one has 
\[\lim\limits_{k \to +\infty}\frac{1}{T_k}h_{\mu_k} = \psi(a).\]
\end{corollary}
\begin{proof}
 This follows immediately from properties (1) and (4) of the previous lemma.
\end{proof}

We now use the fact that for \(\e_0\)-diagonalizable elements with close splittings the singular values and eigenvalues are close and they almost multiply under composition.
\begin{lemma}
  There exists a decreasing positive sequence \(\epsilon_k' \downarrow 0\) such that for all \(k\) large enough, all \(n\) and \(\gamma_1,\ldots, \gamma_n \in C_k\) one has
\[\|\frac{1}{n}(\log\circ S \circ \rho)(\gamma_1\cdots \gamma_n) - T_k a\| \le \epsilon_k' T_k.\]
\end{lemma}
\begin{proof}
  Let \(L(g)\) denote the vector of absolute values of eigenvalues of a diagonalizable element \(g \in \SL(d,\R)\) in decreasing order.

  By \cite[Lemma 2.16]{2021breuillard-sert} there exists \(C > 0\) such that for all \(\e_0/2\)-diagonalizable \(g \in \SL(d,\R)\) one has
  \[\|\log\circ S(g) - \log\circ L(g)\| \le C.\]

  By \cite[Proposition 2.17]{2021breuillard-sert} there exists \(C > 0\) independent of \(k\) and \(n\) such that we have
  \[\| (\log \circ L \circ \rho)(\gamma_1\cdots \gamma_n) - \sum\limits_{i = 1}^n (\log \circ L \circ \rho)(\gamma_i)\| \le Cn.\]

  Combining this with lemma \ref{bklemma} we obtain
  \begin{align*}\|(\log\circ S \circ \rho)(\gamma_1 \cdots \gamma_n) - n T_k a\| &\le C + \|(\log \circ L \circ \rho)(\gamma_1\cdots \gamma_n) - nT_k a\|
    \\ &\le  C(n+1) + \|\sum\limits_{i = 1}^n(\log \circ L \circ \rho)(\gamma_i) - nT_k a\|
    \\ &\le C(2n+1) + \|\sum\limits_{i = 1}^n(\log \circ S \circ \rho)(\gamma_i) - nT_k a\|
    \\ &\le C(2n+1) + n \epsilon_k T_k.
  \end{align*}

  This proves the desired result with \(\epsilon_k' = 3C / T_k + \epsilon_k\).
\end{proof}

The following immediate corollary concludes the proof of proposition \ref{goodmeasureslemma}.

\begin{corollary}
For any sequence of probability measures \(\mu_k\) with \(\mu_k(C_k) = 1\) one has 
\[\lim\limits_{k \to +\infty}\frac{1}{T_k}\lambda(\rho_*\mu_k) = a.\]
\end{corollary}

\section{Covering by balls  and proof of theorem \ref{theoremc}}\label{prooftheoremc}

Fix \( \eta >0\) and set \( s:= \dim _F^P(\rho) + \eta .\) We may assume \( s <\dim (\F _{P}) ,\) otherwise \( \dim _M (\La _\rho ) \leq \dim_F^P(\rho) + \eta \) holds trivially. Since \( s > \dim _F^P( \rho),\) the series 
\( \Phi ^P_\rho (s) := \sum\limits_{\gamma \in \Gamma}\varphi_s^P(\rho(\gamma))\) converges. We are going to construct  covers \( \U\) of \( \La _\rho \) by balls of arbitrarily small radius \( \e\) with less than \(\Phi ^P_\rho(s) \e^{-s+o(\e)}\) elements.  This shows that the ratio \(N(\La _\rho, \e)/\e^{-s +o(\e)} \)  is bounded from above by \( \Phi ^P_\rho (s) \) uniformly in \(\e\). Therefore, \( \dim _M (\La _\rho) \leq s =\dim _F^P (\rho)+\eta\) for all positive \( \eta \) and theorem \ref{theoremc} follows.

\subsection{Shadows and Anosov representations}

 Recall that \(\G\) is a hyperbolic group and that we  chose  the generating set \( \S\) to be symmetric. The distance  \( d(\g,\g') \) on \(\G\) is given by the word length  of  \(\g ^{-1} \g'.\) A {\it{ geodesic }} \( \s = \{\g_n\}_{n \in \Z} \) in \( \G\) is a sequence   such that for all \(( i ,j) , d(\g_i, \g_j ) =| j-i|.\) Any point \( x \in \pp \G\) is the limit point of (at least) one geodesic ray \( \s =\{ \g_j \}_{j \geq 0} \) with \( \g_0 = e.\) 
 
 For a  geodesic ray \( \s =\{ \g_j \}_{j \geq 0} \) with \( \g_0 = e,\) we call the \(R\)-shadow of the geodesic ray the image by \(\xi \) of the set of  limit points of  geodesic rays \(\s ' = \{ \g'_j \}_{j \geq 0} \) satisfying  \( \g'_j = \g_j \) for  \( j\leq R.\) By definition, the \(R\)-shadow of a geodesic ray is a subset of \(\F_{P}.\)

For \((i,j) \in S(P),\g \in \G,\) write \( \zeta _{i,j}(\g): = \log(s_i(\rho(\gamma))/s_j(\rho(\gamma))).\) 
The main step for the proof of theorem \ref{theoremc} is the following proposition:
\begin{proposition}\label{covering} For each  geodesic ray \( \s =\{ \g_j \}_{j \geq 0} \) with \( \g_0 = e,\)  all \( \zeta >0\), the number of balls of radius \( \exp (-\zeta)  \) in \( \F_{P}\) needed to cover the \(R\)-shadow of the geodesic ray is at most \[ \exp \left( \sum _{(i,j) \in S(P)} [  \zeta - \zeta _{i,j}(\g_R) ]^+ + o(R) \right),\] where, for a real \( \varpi,\)  \( \varpi^+ = \max \{ \varpi, 0\} .\) \end{proposition}

\subsection{Proof of Theorem \ref{theoremc} assuming proposition \ref{covering}}

  Fix \( \e >0 \) small. We need  to cover \(\La_\rho \) by well-chosen shadows and then cover these shadows by \(\e\)-balls.

For \( \g \in \G,\) we write the \( \zeta _{i,j}(\g) , (i,j) \in S(P) ,\) in nondecreasing order as  \( 0< \zeta _1(\g) \leq \zeta_2 (\g) <\ldots \leq \zeta _{\# S(P)} (\g).\) For \(\eta >0,\) write \( s:= \dim _F^P (\rho) + \eta .\) We may assume \( s<\#S(P) \) and let \( q\) be  a positive integer such  that \( q-1 \leq s \leq q .\) For any geodesic ray \( \s =\{ \g_j \}_{j \geq 0}, \)  we will use proposition \ref{covering} with \( \zeta = \zeta _q(\g_R) \) to estimate the \(\e\)-covering number of its \(R\)-shadow.

Let \( \s =\{ \g_j \}_{j \geq 0} \) be a geodesic ray such that \( \g_0(x) = e.\) Then, the sequence \( \{\zeta _q (\g_n)\}_{n\geq 0} \) diverges to infinity, has bounded gaps and there are   \( C,K >0 \) such that \(\zeta_q (\g_{n+K}) - \zeta _q (\g_n) >C\) (see \cite[theorem 1.3]{GGKW17}). It follows that for any chosen geodesic ray \( \s \), there is a well defined smallest index \( n(\rho,\s)\) and \( C>1\) such that 
\begin{equation}\label{choiceofR} \log (1/\e) \; \leq  \; \zeta_q((\g_{n(\rho,\s)}))  \; < \;  \log(1/\e) +C.\end{equation}

 By proposition \ref{covering} applied with \(R = n(\rho, \s)  \) and  \( \zeta = \zeta _q(\g_{n(\rho,\s)}) \),  we can cover the \(R\)-shadow of \( \g_{R} \) with less than 
 \[ \exp \left(  \sum _{(i,j) \in S(P)} [  \zeta_q(\g_R)- \zeta _{i,j}(\g_R) ]^+  + o(\log (1/\e)) \right)\] 
 balls of radius \( \e.\) We claim  that 
\begin{equation*}   \sum _{(i,j) \in S(P)} [  \zeta _q(\g_R)- \zeta _{i,j}(\g_R) ]^+  \;\leq \; - \min \sum\limits_{(i,j) \in S(P)}c_{i,j}\zeta_{i,j}(\g_R)
 +s  \log(1/\e) +C, \end{equation*}  where  the minimum is over \(0 \le c_{i,j} \le 1\) with \(\sum c_{i,j} = s\).
 
 Indeed, since we have ordered the values \(\zeta _k(\g_R)\) in nondecreasing order, the above minimum is attained  for 
 \[ c_k = 1 {\textrm{ for }} k < q, \; c_q = s-q+1, \;  c_k = 0 {\textrm{ for }} k >q .\]
 With that choice of \( c_k\)'s, we have 
  \[  \sum _{(i,j) \in S(P)} [  \zeta _q(\g _R) - \zeta _{i,j}(\g_R) ]^+ = -\left(\sum _k c_k \zeta _k(\g_{R})  +c_q \zeta_q(\g_{R})\right) +s \zeta_q(\g_{R})\] 
  and the claim follows from (\ref{choiceofR}).
  
 Cover now \( \La_\rho\) by  \(n(\rho,\s)\)-shadows of distinct \( \g_{n(\rho,\s)} \). As announced, this proves that   \[ N( \La _\rho, \e) \; \leq \;  \sum\limits_{\gamma \in \Gamma} \exp (-(F_s^P\circ\log\circ S \circ \rho)(\gamma)) \e^{-s+o(\e)} =  \Phi ^P_\rho (s)\e^{-s+o(\e)} .\]

\subsection{Geometry of \(\F_{P}\)}

We recall the description in \cite{LL23} of the geometric structure of the successive Lipschitz foliations by Euclidean spaces on \( \F_{P}.\) Write \( P = \{  p_1, \ldots, p_{M} \} \). Recall that, by convention, \( p_{M+1} = d.\)

Recall that a  topology on \( \{ 1\ldots, M +1\}\) is called admissible if the subsets \( \{ i, i+1,\ldots, M+1\}\) are open. An admissible topology is described by its atoms \(T(i)\), where \(T(i)\) is the smallest open set containing \( \{i\} .\) We write \(T_0\) for the topology with atoms \( T(i) =\{ i,i+1, \ldots, M+1\}\), \(T_P\) for the topology with atoms \( T(i) =\{ i\}\). An admissible topology \(T\) is finer than another one \(T'\) (denoted \( T \prec T'\)) if any \(T' \)-open set is \(T\)-open. By definition, any admissible topology is finer that \(T_0\). 

Given an admissible topology \(T\), we define the (weighted) configuration space \(\X_T\) (with weights \(p_1,p_2-p_1,\ldots,d-p_{M}\)) as the space of sequences \((x_I)_{I \in T}\) such that
\begin{enumerate}
 \item \(x_I\) is a \(\sum _{i\in I}(p_i -p_{i-1})\)-dimensional subspace of \(\R^d\) for each \(I \in T\),
 \item \(x_{I \cup J} = x_I + x_J\) for all \(I,J \in T\), and
 \item \(x_{I\cap J} = x_I \cap x_J\) for all \(I,J \in T\).
\end{enumerate}

Each configuration space \(\X_T\) is endowed with the distance corresponding to its natural embedding in the product of Grassmannian manifolds.

For \( T \prec T',\) there is a natural projection \(\pi_{T,T'} : \X_T \to \X_{T'} \). The space \( \X_{T_P} \) is identified with the pairs in \( (\F_{P}, \F_{P^\ast} )\) in general position. In particular, given \(y \in \F_{P^\ast} \) the projection  \( \pi _{T_P, T_{0}} \) is of the form  \( \pi ^y \x \id,\) where  \( \pi ^y \) is the natural projection from the set of flags in \(\F_{P}\) in general position with \(y\) to  a point.  The fibers \( ( \pi ^y )^{-1} (y)\) are \( \# S(P) \)-dimensional  open subsets of \(\F_{P}.\) 

 We say a sequence of subspaces \(V = (V_1,\ldots, V_{M+1})\) is a splitting compatible with \(y \in \F_{P^\ast} \) if for all \(i, 1\leq i \leq M+1,\) 
\begin{equation}\label{compatiblesplittingequation}
  y_{\{j: j\geq {i}\}} = \bigoplus\limits_{\{j: j\ge i\}}V_j.
\end{equation}

Notice that in particular this implies \(\dim(V_i) = p_i - p_{i-1}\) for all \(i\) and \(\R^d = \bigoplus\limits_{i = 1}^{M+1} V_i\).
In \cite{LL23} lemma 4.1,  we show that setting
\begin{equation}\label{perpcompatiblesplitting}
V_i(y) = y_{\{j: j\geq {i}\} \cap \left(y_{\{j: j>i\}} \right)^{\perp}},
\end{equation}
 yields a compatible splitting for each \(y \in \F_{P^\ast} \) that we call the perpendicular splitting compatible with \(y\).

 Given \(y \in \F_{P^\ast} \) and \(V\) a splitting compatible with \(y\), we denote by \(\nil (V)\) the space of linear mapping \(f:\R^d \to \R^d\) such that
\begin{equation}\label{defineVequation}
 f\left(V_{M+1}\right)  \;= \{0\} \; {\textrm { and }} \;  f\left(V_i\right) \subset \bigoplus\limits_{j : j>i } V_j,
\end{equation}
for \(i = 1,\ldots,M\).  We have \( \dim \nil(V) = \# S(P ).\)  Given \(y\in \F_{P^\ast} \) and a compatible splitting \(V\) we define a mapping
\[\varphi_V: \nil(V) \to \F^{y}\]
by setting
\begin{equation}\label{lineartoconfiguration}
  \varphi_{V}(f)_I = \left(\id + f\right)\left(\bigoplus\limits_{i \in I}V_i(y)\right),
\end{equation}
for all \(I \in T_P\), where \(\id:\R^d \to \R^d\) is the identity mapping.

For each \(y \in \F_{P^\ast} \)  we consider the perpendicular compatible splitting \(V(y) = (V_1(y),\ldots, V_{M+1}(y))\) and we define
\begin{equation}\label{vectorbundleequation}
  \V = \lbrace (y,f):  y \in \F_{P^\ast} ,   f \in \nil(V)\rbrace.
\end{equation}

This is a vector bundle with base \(\F_{P^\ast} \) given by the projection onto the first coordinate.  It is a sub-bundle of the product \(\F_{P^\ast}  \times \text{Hom}(\R^d,\R^d)\).  We endow it with the metric given by  the sum of the distance in \(\F_{P^\ast} \) and the Hilbert-Schmidt norm  on \(\text{Hom}(\R^d,\R^d)\) associated with the Euclidean structures on \(V_i, V_j.\) We have

\begin{theorem}\label{fiberbundle}The mapping  \(\varphi :\V \to \F_{P} \) defined by
\[\varphi (y,f) = \varphi_{V(y)}(f),\]
is a locally bilipschitz homeomorphism. \end{theorem}
\begin{proof}This follows from \cite{LL23} theorem 2.4 and its proof (\cite{LL23}, section 4). \end{proof}

\begin{theorem}\label{changeofsplitting} Let \(y\in \F_{P^\ast} \) and \(V,W\) be two splittings compatible with \(y\). Then, the mapping
  \[\varphi_W^{-1}\circ \varphi_{V}:\nil(V) \to \nil (W),\]
  is affine.
\end{theorem}
\begin{proof} See \cite{LL23},  lemma 5.8. 
\end{proof}

Let \( g \in SL(d,\R )\).  We note that if \(V = (V_1,\ldots,V_{M+1})\) is a splitting compatible with \(y\) then \(g^{-1}V = (g^{-1}V_1,\ldots, g^{-1}V_{M+1})\) is a splitting compatible with  \(g^{-1}y\).
For the coordinates given by these two splittings the action of \(g\)  is linear between the corresponding fibers:

\begin{theorem}[Linearizing coordinates]\label{linearactionlemma}
  For each \(y\in \F_{P^\ast} \) and each adapted splitting \(V\) one has
  \[\varphi_{V}^{-1}\circ g \circ \varphi_{g^{-1}V}(f) = g fg^{-1},\]
  for all \(g \in G\) and all \(f \in \nil (g^{-1}V)\).

  In particular, \(\varphi_{V}^{-1}\circ g\circ \varphi_{g^{-1}V}:\nil(g^{-1}V) \to \nil(V)\) is linear and the action of \( SL(d,\R) \) on \(\X_{T_P} \) is affine on each fiber.
\end{theorem}
  \begin{proof} See \cite{LL23} lemma 6.2 and corollary 6.3.\end{proof}

  \subsection{Proof of proposition \ref{covering}}
Let \( \g \neq e  \in \G\). Recall that we defined   \(\xi^p(\gamma) \subset \R^d\) as the unique \(p\)-dimensional subspace on which \(p\)-dimensional volume is most contracted by \(\rho(\gamma)^{-1}\).
Since the representation \( \rho\) satisfies the \(P^\ast\)-Anosov condition, we can associate to  the matrix \( \rho(\g)\) the flag 
 \[ \xi^\ast(\g) := (\{0\}, \xi^{d-p_{M}}(\gamma ^{-1}),\ldots,\xi^{d-p_1}(\gamma ^{-1}), \R^d) \in \F_{P^\ast} .\] 
From the definitions we obtain:
\begin{lemma}\label{canonicalsplitting} The sequence of subspaces \(V = (V_1,\ldots, V_{M+1})\) is a splitting compatible with \(y =\xi ^\ast (\g_R) \) where \( V_j\) is the sum of the eigenspaces of \( \sqrt {\rho(\g_R) (\rho( \g_R))^t} \) corresponding to the singular values \( s_i \) for \( p_{j-1} <i \leq p_j.\) \end{lemma}

Let \(y =\xi ^\ast (\g_R) \). Using the coordinates on the space \( \F_P^y\) given by theorem \ref{fiberbundle} with the splitting compatible with \(y\) given by lemma \ref{canonicalsplitting}, we can write 
\begin{lemma}\label{angle} Let \(y =\xi ^\ast (\g_R) \) and \(V\) the splitting compatible with \(y\) given by lemma \ref{canonicalsplitting}. Write \(\varphi_V: \nil_{T_P,T_0}(V) \to \X_{T_P,T_0}^{y} = \F_P^y\) for the coordinate mapping given by (\ref{lineartoconfiguration}). Let \(x \) belong to the \(R\)-shadow of the geodesic  ray  \( \s =\{ \g_j \}_{0\leq j \leq R} \). Then, there exists \(K , \tau >0\) such that for \(R \geq K, \) \begin{equation*}  \| \varphi _V ^{-1} (\g_R^{-1}x) \|\; < \;  \tau . \end{equation*} \end{lemma}
\begin{proof}  Since \(x \) belong to the \(R\)-shadow of the geodesic  ray  \( \s =\{ \g_j \}_{0\leq j \leq R} \) there is a geodesic  ray  \(\s' (x)=:\{ \g'_j \}_{j \geq 0} \) such that \( \g'_0= e, \g' _R  = \g _R\)  and \( \xi (\s') = x.\) Applying \(\g_R^{-1}\), there is a geodesic ray \(\s''(x) =\g_R^{-1} \s'(x) =: \{ \g''_j \}_{j \geq 0} \) such that \( \g''_0 = \g_R^{-1}, \g'' _R  = e\)  and \( \xi (\s'') = \g_R^{-1} x.\) By the proof of theorem \ref{boundarymapstheorem}, if \(R\) is large enough, there is \(\tau_0\) such that 
\[ \dist (\xi (V), \g_R^{-1} x ) \; \le \;\tau _0,\]
where \(\xi (V) \in \F_P \) is given by \( \xi (V) := {0} \subset V_1 \subset V_1 \oplus V_2 \subset \ldots \subset \R^d \).  Using theorems \ref{fiberbundle} and \ref{changeofsplitting}, the lemma folllows. \end{proof}
 
We can now prove proposition \ref{covering}: let  \( \s =\{ \g_j \}_{j \geq 0} \) be a geodesic ray with \( \g_0 = e.\)  Let \(V\) the splitting compatible with \(\xi ^\ast (\g_R) \) given by lemma \ref{canonicalsplitting}. By lemma \ref{angle}, if \(R >K, \) the \( R \)-shadow of \( \s\) is contained in \( \varphi _{\rho(\g_R) V}^{-1}\circ \rho(\g_R) \circ \varphi _V (B(0, \tau )),\) where \(B(0, \tau ) \) is  the ball of radius \( \tau  \) in \( \nil_{T_P,T_0}(V).\) Proposition \ref{covering} amounts to the following 
\begin{lemma} With the preceding notations, the image of \(B(0, \tau ) \) by \( \varphi _{\rho(\g_R )V}^{-1}\circ \rho(\g_R ) \circ \varphi _V \) is an ellipsoid with axes 
\( \tau  \exp (\zeta _{i, j }(\g _R)),\)  for all \( (i,j) \in S(P).\) \end{lemma}
\begin{proof} By theorem \ref{linearactionlemma},
\[ \varphi _{\rho(\g_R) V}^{-1}\circ \rho(\g_R) \circ \varphi _V (f) \; = \; \rho(\g_R) f \rho(\g_R)^{-1} .\] 
Write \( \nil_{T_P,T_0}(V)\) as  \( \bigoplus _{1\le i<j \le M+1} \text{Hom}(V_i,V_j)\) and \( f \in B(0, \tau) \) as \( f = \{ f_{i, j } \}_{1\le i<j \le M+1}, \) with all \( |f_{i,j} |< \tau.\) The matrix \(\rho( \g_R) \) is block diagonal, made of matrices \(g^i  \in \text{Hom}(V_i,V_i)\). Therefore, \( \rho(\g_R) f \rho(\g_R)^{-1} = \{ g^jf_{i,j} (g^i)^{-1}, 1\le i<j \le M+1\} .\)

 There is an orthonormal basis of \(V_i,\) namely \(\{v_1, \ldots , v_{p_i - p_{i-1}}\},\) such that \( (g^i )^{-1}v_\ell, \ell = 1, \ldots , p_i -p_{i-1}\)  form an orthogonal system with \( \| (g^i)^{-1} v_\ell \| = (s _{i-1+\ell}(\rho (\g_R)))^{-1} .\) Similarly, there is an orthonormal basis of \(V_j,\) namely \(\{u_1, \ldots , u_{p_j - p_{j-1}}\},\) such that \( g^j u_k, k = 1, \ldots , p_j -p_{j-1}\)  form an orthogonal system with \( \| g^j u_k \| = s _{j-1+k} (\rho (\g_R)) .\) 
 
  For \( 1\leq \ell \leq p_i - p_{i-1}, 1\leq k \leq p_j -p_{j-1},  \) write \( f^{ \ell ,k } \) for the element of \( \text{Hom}(V_i,V_j)\) that sends \(  (g^i )^{-1}v_\ell/  \|(g^i )^{-1}v_\ell \|  \) to \( u_k \) and the orthogonal space \(((g^i )^{-1}v_\ell)^\perp \) to \( 0 \). For all \(  \ell, k,0< \ell \le p_i- p_{i-1} , 0< k \leq p_j - p_{j-1}, \) the \( f^{\ell ,k} \) form   an orthogonal basis of \(  \text{Hom}(V_i,V_j)\) such that the \( \varphi _{\rho(\g_R )V}^{-1}\circ \rho(\g_R) \circ \varphi _V (f ^{\ell, k} ) \) are orthogonal with norm \( ( s _{i-1+\ell}(\rho (\g_R)) )^{-1} s _{j-1+k} (\rho (\g_R)) = \exp (-\zeta _{i-1+k, j-1+\ell }(\g _R) ).\)  The lemma follows by putting all the \( \text{Hom}(V_i,V_j)\)  together. \end{proof}
  
  \appendix

\section{Quasi-geodesic semi-groups}

\subsection{Goal}

If \(\Gamma\) is the free group with generators \(a,b\) then the semi-group \(S\) consisting of reduced words which begin and end with \(a\) has the following nice properties:
\begin{enumerate}
 \item \(S\) generates \(\Gamma\) as a group.
 \item For every \(n\) and \(\gamma_1,\ldots, \gamma_n \in S\) the geodesic joining the neutral element \(e\) to \(\gamma_1\cdots \gamma_n\) passes through all partial products \(\gamma_1\cdots \gamma_k\) for \(k = 1,\ldots, n\).
 \item For every \(\gamma \in \Gamma\) there exists two letter words \(L(\gamma)\) and \(R(\gamma)\) such that \(L(\gamma)\gamma R(\gamma) \in S\).
\end{enumerate}

Our goal in this appendix is to give a similar construction for a general hyperbolic group.

For this purpose we fix a group \(\Gamma\) with a finite symmetric generating set \(\S\).

We let \(|\gamma|\) be the word length of \(\gamma \in \Gamma\) with respect to \(\S\) and consider the word distance \(\dist(\gamma,\eta) = |\gamma^{-1}\eta|\).

We assume that there is \(\delta > 0\) (fixed from now on) such that \(\Gamma\) is \(\delta\)-hyperbolic with respect to \(\dist\), and non-elementary (i.e. contains two independent loxodromic elements).

Recall that given positive constants \(C,D > 0\) a \((C,D)\)-quasi-geodesic is a sequence \(x: I \cap \Z \to \Gamma\) where \(I\) is an interval, satisfying
\[C^{-1}|m-n| - D \le \dist(x_m,x_n) \le C|m-n| + D,\]
for all \(m,n \in I \cap \Z\).

We say a semigroup \(S \subset \Gamma\) is quasi-geodesic if there exist constants \(C,D > 0\) such that for all \(n\) and \(\gamma_1,\ldots, \gamma_n \in S\) the sequence
\[e, \gamma_1,\gamma_1\gamma_2,\ldots, \gamma_1\cdots \gamma_n,\]
is visited in left to right order by some \((C,D)\)-quasi-geodesic.

We now state the main result of the appendix.

\begin{proposition}\label{semigrouplemma}
 Let \((\Gamma,\S)\) be a non-elementary hyperbolic group with a finite symmetric generator \(\S\).   Then there exists a semigroup \(S \subset \Gamma\), a finite set \(F \subset \Gamma\) and mappings \(L,R:\Gamma \to F\) with the following properties:
 \begin{enumerate}
  \item \(S\) generates \(\Gamma\) as a group,
  \item \(S\) is quasi-geodesic, and
  \item \(L(\gamma)\gamma R(\gamma) \in S\) for all \(\gamma \in \Gamma\).
 \end{enumerate}
\end{proposition}

In the case \(\Gamma\) is the free group generated freely by \(\S = \lbrace a,b,a^{-1},b^{-1}\rbrace\) we may set \(S\) to be the semigroup of elements beginning and ending in \(a\) in reduced form, \(L(\gamma)\) to be \(a\) if the reduced form of \(\gamma\) does not begin with \(a^{-1}\) and \(ab\) otherwise, and \(R(\gamma)\) similarly to be \(a\) if the last letter of \(\gamma\) is not \(a^{-1}\) and \(ba\) otherwise.

The arguments in what follows use the local-to-global principle for hyperbolic groups as in \cite[Section 3]{gouezel}.

\subsection{Proof of proposition \ref{semigrouplemma}}

Recall that the Gromov product (based at the neutral element \(e \in \Gamma\)) is defined by
\[(\gamma,\eta) = \frac{|\gamma| + |\eta| - |\gamma^{-1}\eta|}{2}.\]

We assume \(\delta > 0\) is such that
\[\min\lbrace (\gamma_1,\gamma_2),(\gamma_2,\gamma_3)\rbrace \le (\gamma_1,\gamma_3) + \delta,\]
for all \(\gamma_1,\gamma_2,\gamma_3 \in \Gamma\).  We further assume that
\[\dist(e,\alpha) \le (p,q) + \delta,\]
for all \(p,q \in \Gamma\) and geodesic \(\alpha\) joining them (such a choice of \(\delta\) is possible by \cite[Proposition 1.22]{bridson-haefliger}).

\begin{lemma}
 For each \(C \ge 1\) there exist \(a,b \in \Gamma\) such that setting \(A = \lbrace a,b,a^{-1},b^{-1}\rbrace\) and defining
 \[c_1 = \max\limits_{x,y \in A:\ x \neq y}(x,y)\text{ and }c_2 = \min\limits_{x \in A}|x|,\]
 one has \(c_1 + 2\delta + r < \frac{1}{2}{c_2}\) for some \(r > C(c_1+2\delta+1)\).
\end{lemma}
\begin{proof}
 Since \(\Gamma\) is non-elementary we may take \(x,y\) two independent loxodromic elements (i.e. the set of boundary fixed points of \(x\) and \(y\) are disjoint).
 
 Letting \(a = x^n\) and \(b = y^n\) for sufficiently large \(n\) establishes the claim since \(c_1(x^n,y^n)\) remains bounded, while \(c_2(x^n,y^n)\) goes to infinity.
\end{proof}

We fix from now on a set \(A\) as in the previous lemma with corresponding constants \(c_1,c_2,r\) satisfying
\begin{equation}\label{requation}
r > 32(c_1+2\delta+1). 
\end{equation}

\begin{lemma}\label{borrowedlemma}
 The sets defined for \(x \in A\) by \(V(x) = \lbrace \gamma: (\gamma,x) \ge c_1+2\delta + r\rbrace\) are pairwise disjoint.  
 
 For all \(x \in A\) and \(\gamma \notin V(x^{-1})\) one has \(x\gamma \in V(x)\).  
 
 For all \(x,y \in A\) with \(x \neq y\), if \(\gamma \in V(x)\) and \(\eta \in V(y)\) then \((\gamma,\eta) \le c_1+2\delta\).
 
 If \(\gamma \in V(x)\) for some \(x \in A\) then \(|\gamma| \ge r\). 
\end{lemma}
\begin{proof}
By hypothesis \(2\delta + r \in (2\delta,c_2/2 - c_1)\).  With this property, the first three statements are claims 1-3 in the proof of \cite[Lemma 4.1]{random}.

For the last statement observe that
\[|\gamma| = (\gamma,\gamma) \ge \min\lbrace (\gamma,x),(x,x)\rbrace - \delta \ge \min \lbrace c_1 + 2\delta + r, c_2\rbrace - \delta \ge r.\]
\end{proof}

\begin{lemma}\label{mlemma}
 There exists \(m\) such that for all \(\gamma \in V(a)\) and all \(\eta\) with \(\dist(\gamma,\eta) \le 2c_2\) one has \(a^m\eta \in V(a)\).
\end{lemma}
\begin{proof}
 We first observe that by lemma \ref{borrowedlemma} we have \((a^{-1},a) \le c_1 + 2\delta\) and therefore a geodesic fixed under multiplication by \(a\) is at distance at most \(c_1+3\delta\) from \(a^m\) for all \(m\).
 
 It follows, letting \(\ell = \lim\limits_{n \to +\infty}\frac{1}{n}|a^n|\) be the translation length of \(a\) that
 \[\ell \ge |a| - 2(c_1+3\delta) \ge 2(c_1+2\delta+r) -2(c_1+3\delta) \ge 2r-2\delta.\]
 
 Hence we obtain for all \(m \ge 1\) that
 \begin{align*}
(a,a^m) &= \frac{|a|+|a^m|-|a^{m-1}|}{2} 
\\ &\ge \frac{\ell -2(c_1+3\delta) + m\ell - 2(c_1+3\delta) - (m-1)\ell -2(c_1+3\delta)}{2}
\\ &= \ell - 3(c_1+3\delta)
\\ &\ge  2r - 3c_1-11\delta
\\ &\ge  c_1+3\delta + r.
 \end{align*}
 
In what follows we will need to switch basepoints so we recall that the Gromov product
\[(x,y)_z = \frac{\dist(x,z)+\dist(y,z)-\dist(x,y)}{2},\]
satisfies
\((x,y)_z + (z,y)_x = \dist(x,z)\).

We now calculate using lemma \ref{borrowedlemma}
 \begin{align*}
(a^m,a^{m}\eta) &= (e,\eta)_{a^{-m}} 
\\ &= |a^m| - (a^{-m},\eta)_{e} 
\\ &\ge |a^m| - (a^{-m},\gamma)_{e} - 2c_2 
\\ &\ge |a^m| - c_1 - 2\delta - 2c_2 
\\ &\ge c_1+3\delta + r,
 \end{align*}
 if we choose \(m\) large enough depending only on \(c_1,c_2\) and \(\delta\).
 
 We now conclude the proof using the hyperbolicity property since
 \[(a,a^m\eta) \ge \min\lbrace (a,a^m), (a^m,a^m\eta) \rbrace -\delta \ge c_1 + 2\delta + r,\]
 so \(a^m\eta \in V(a)\) as required.
\end{proof}

We define
\[T = \lbrace \gamma: \gamma \in V(a)\text{ and }\gamma^{-1} \in V(a^{-1})\rbrace.\]

With \(m\) given by the previous lemma, for each \(\gamma \in \Gamma\) let
\[L(\gamma) = \left\lbrace\begin{array}{ll}a^{m+1} &\text{ if } \gamma \notin V(a^{-1}),\\ a^{m+1}b &\text{ otherwise,}\end{array}\right.\]
and
\[R(\gamma) = \left\lbrace\begin{array}{ll}a &\text{ if } (L(\gamma)\gamma)^{-1} \notin V(a),\\ ba &\text{ otherwise.}\end{array}\right.\]

\begin{lemma}
 One has \(L(\gamma)\gamma R(\gamma) \in T\) for all \(\gamma \in \Gamma\).
\end{lemma}
\begin{proof}
Since \(\dist(L(\gamma)\gamma, L(\gamma)\gamma R(\gamma)) \le 2c_2\) this follows immediately from lemma \ref{mlemma}. 
\end{proof}

We now let \(S\) be the semigroup generated by \(T\).

\begin{lemma}
 The semigroup \(S\) generates \(\Gamma\) as a group.
\end{lemma}
\begin{proof} 
 From the definition \(a \in S\) and  (using lemma \ref{borrowedlemma}) \(aba \in S\).  Therefore \(L(\gamma),R(\gamma)\) belong to the group generated by \(S\) for all \(\gamma\).
 
 The equation
 \[\gamma = L(\gamma)^{-1}\left( L(\gamma)\gamma R(\gamma) \right)R(\gamma)^{-1}\]
 now implies that \(S\) generates \(\Gamma\) as a group, as required.
\end{proof}

As a first step towards showing that \(S\) is quasi-geodesic we prove the following.

\begin{lemma}\label{boundeddistanelemma}
For any \(n\) and \(\gamma_1,\ldots, \gamma_n \in T\), letting \(\alpha\) be a geodesic joining \(e\) and \(\gamma_1\cdots \gamma_n\), one has
\[|\gamma_1\cdots\gamma_n| \ge |\gamma_1\cdots \gamma_k| + |\gamma_{k+1}\cdots \gamma_n| - r/8,\]
amd \(\dist(\gamma_1\cdots \gamma_k,\alpha) \le r/8\) for all \(k = 0,\ldots, n\).
\end{lemma}
\begin{proof}
Setting \(x_k = \gamma_1\cdots \gamma_k\) for \(k = 0,\ldots, n\) we observe that
\[(x_{k-1},x_{k+1})_{x_k} = (\gamma_{k}^{-1},\gamma_{k+1}) \le c_1 + 2\delta,\]
and \(\dist(x_k,x_{k+1}) = |\gamma_{k+1}| \ge r\) by lemma \ref{borrowedlemma}.

Hence the sequence is a \((c_1+\delta,r)\)-chain as in \cite{gouezel}.

We observe that
\[r > 2(c_1 + 2\delta) +4\delta + 1,\]
and therefore by \cite[Lemma 3.8]{gouezel} we have
\[(x_0,x_n)_{x_k} \le c_1 + 4\delta,\]
for all \(k\).

This implies that
\[\dist(x_k,\alpha) \le c_1 + 5\delta < r/8.\]

Also since \((x_0,x_n)_{x_k} = (x_k^{-1},x_k^{-1}x_n)_e\) we get by definition
\[\frac{|\gamma_1\cdots \gamma_k| + |\gamma_{k+1}\cdots \gamma_n| - |\gamma_1\cdots \gamma_n|}{2} \le c_1 + 4\delta \le r/16,\]
which yields the required lower bound for \(|\gamma_1\cdots \gamma_n|\).
 \end{proof}

We now conclude the proof of proposition \ref{semigrouplemma}.

\begin{lemma}
The semigroup \(S\) is quasi-geodesic.
\end{lemma}
\begin{proof}
Fix \(n\) and \(\gamma_1,\ldots, \gamma_n \in T\), and let \(\alpha\) be a geodesic with \(\alpha_0 = e\) and \(\alpha_{N_n} = \gamma_1\cdots \gamma_n\) with \(N_n = |\gamma_1\cdots \gamma_n|\).

Since the word distance is integer valued setting \(C = \lfloor r/8 \rfloor\) we have by lemma \ref{boundeddistanelemma} that there exist \(N_0 = 0, N_1, \ldots,N_n = |\gamma_1\cdots \gamma_n|\) such that,
\[\dist(\gamma_1\cdots \gamma_k, \alpha_{N_k}) \le C \le r/8,\]
for all \(k\).

We also obtain from lemma \ref{boundeddistanelemma} applied to \(\gamma_1,\ldots, \gamma_{k+1}\) and lemma \ref{borrowedlemma} that
\begin{align*}
N_{k+1} - N_k &\ge |\gamma_1\cdots \gamma_{k+1}|-r/8 - |\gamma_1\cdots \gamma_k| -r/8
\\ &\ge |\gamma_{k+1}| - 3r/8 \ge 5r/8 \ge 5C,
\end{align*}
and in particular \(N_{k} \le N_{k+1}\) for \(k = 0,\ldots, n-1\).

We now construct a path \(\beta\) which we claim to be quasi-geodesic (with constants independent of \(n\) and \(\gamma_1,\ldots,\gamma_n\)) by concatenating segments of \(\alpha\) with a paths to and from each \(\gamma_1\cdots \gamma_k\) (which will have length at most \(2C\) as seen above).   

To be precise we choose \(\beta:[0,N_n + 2(n-1)C] \cap \Z \to \Gamma\) such that
\begin{enumerate}
 \item For each \(k = 0,\ldots,n-1\) one has \(\beta_{m+N_k+2kC} = \alpha_{m+N_k}\) for all \(m \in [0,N_{k+1}-N_k] \cap \Z\).
 \item \(\beta_{N_k + 2(k-1)C + C} = \gamma_1\cdots \gamma_k\) for \(k = 1,\ldots, n-1\).
 \item \(\dist(\beta_m, \beta_{m+1})\le 1\) for all \(m\).
\end{enumerate}

In view of the third property above we have
\(\dist(\beta_l,\beta_m) \le |l-m|\).

For the lower bound given \(0 \le l < m \le N_n + 2(n-1)C\) we write
\[\dist(\beta_l,\beta_m) = \sum\limits_{k = 0}^{m-l}\dist(\beta_{l+k},\beta_{l+k+1}) \ge \sum\limits_{k = 0}^{m-l}f(k),\]
where we set \(f(k) = 1\) if \(k \in [N_i+2iC,N_{i+1}+2iC]\) for some \(i\) and \(f(k) = -1\) otherwise.

Since the sequence \(f(0),\ldots,f(l-m)\) consists of subsequences of runs of at most \(2C\) consecutive times the value \(-1\), with runs of at least \(5C\) times the value \(1\) in between, we obtain 
\[\dist(\beta_l,\beta_m) \ge 3(|m-l| - 4C)/7 - 4C,\]
which establishes that \(S\) is \((7/3, 6C)\)-quasi-geodesic.
\end{proof}

\newcommand{\etalchar}[1]{$^{#1}$}

\end{document}